\documentclass[11pt]{amsart}

\usepackage{amsmath, amsthm, amssymb}

\usepackage{lmodern}
\usepackage{palatino}                       
\usepackage[bigdelims,vvarbb]{newpxmath}    
\usepackage[scaled=0.95]{inconsolata}       
\usepackage[T1]{fontenc}                    

\usepackage{comment}

\usepackage[abs]{overpic}		  
\usepackage{import}
\usepackage{transparent}

\usepackage{xcolor}
\definecolor{lred}{RGB}{226, 106, 106}
\definecolor{nred}{RGB}{237, 28, 36}
\definecolor{lblue}{RGB}{52, 152, 219}
\definecolor{nblue}{RGB}{0, 174, 239}
\definecolor{lyellow}{RGB}{232, 197, 91}
\definecolor{dgreen}{RGB}{0, 148, 68}
\definecolor{l1yellow}{RGB}{217, 224, 33}
\definecolor{lgrey}{RGB}{179, 179, 179}
\definecolor{indigo}{rgb}{0.29, 0.0, 0.51}  
\usepackage[colorlinks, urlcolor=indigo, linkcolor=indigo, citecolor=indigo]{hyperref}

\usepackage[hcentering, vcentering, total={5.9in, 8.2in}]{geometry}  

\usepackage{tikz}
\usetikzlibrary{trees}

\theoremstyle{plain}
\newtheorem{theorem}{Theorem}
\newtheorem{corollary}[theorem]{Corollary}
\newtheorem{proposition}[theorem]{Proposition}
\newtheorem{lemma}[theorem]{Lemma}

\newtheorem{conjecture}[theorem]{Conjecture}

\theoremstyle{definition}

\theoremstyle{remark}
\newtheorem{remark}[theorem]{Remark}

\numberwithin{theorem}{section}

\newcommand{\dfn}[1]{{\em #1}}        
\newcommand{\Q}{\mathbb{Q}}           
\newcommand{\Z}{\mathbb{Z}}           
\newcommand{\N}{\mathbb{N}}           
\DeclareMathOperator{\bd}{\partial}   
\newcommand{\PD}{\mathrm{PD}\,} 

\makeatletter
\newcommand*\bigcdot{\mathpalette\bigcdot@{0.6}}
\newcommand*\bigcdot@[2]{\mathbin{\vcenter{\hbox{\scalebox{#2}{$\m@th#1\bullet$}}}}}
\makeatother

\newcommand{\vects}[2]{\left(\begin{smallmatrix} #1 \\ #2 \end{smallmatrix}\right)}  

\newcommand{\matrixp}[4]{{\begin{pmatrix}#1 & #2 \\ #3 & #4\end{pmatrix}}}                  

\DeclareMathOperator\tb{tb}                               
\DeclareMathOperator\rot{rot}                             
\DeclareMathOperator\tw{tw}                               

\DeclareMathOperator{\HFhat}{\widehat{{HF}}}  

\DeclareFontFamily{U} {cmr}{}
\DeclareFontShape{U}{cmr}{m}{n}{
  <-6> cmr5
  <6-7> cmr6
  <7-8> cmr7
  <8-9> cmr8
  <9-10> cmr9
  <10-12> cmr8
  <12-> cmr9}{}
\DeclareSymbolFont{Xcmr} {U} {cmr}{m}{n}
\DeclareMathSymbol{\Phi}{\mathord}{Xcmr}{8}

\begin{document}

\title{Tight contact structures on a family of hyperbolic L-spaces}

\author{Hyunki Min}
\address{Department of Mathematics \\ University of California \\ Los Angeles, CA}
\email{hkmin27@math.ucla.edu}

\author{Isacco Nonino}
\address{Department of Mathematics \\ University of Glasgow \\ Glasgow, Scotland}
\email{Isacco.Nonino@glasgow.ac.uk}

\begin{abstract}
We classify tight contact structures on various surgeries on the Whitehead link, which provides the first classification result on an infinite family of hyperbolic L-spaces. We also determine which of the tight contact structures are Stein fillable and which are virtually overtwisted.
\end{abstract}

\maketitle

\section{Introduction}

A cooriented contact 3-manifold $(M,\xi)$ is {\it overtwisted} if it contains an overtwisted disk, which is an embedded disk tangent to the contact planes along the boundary, and {\it tight} if it does not contain an overtwisted disk. Since Eliashberg \cite{Eliashberg:overtwisted} classified overtwisted contact structures, more recent studies focus on the classification of tight contact structures.

Since tight contact structures respect the prime decomposition of $3$-manifolds by the work of Colin \cite{Colin:prime} (see also \cite{DG:decomp, Honda:gluing}), it is sufficient to consider the classification of tight contact structures on prime manifolds. The geometrization of $3$-manifolds implies that a prime $3$-manifold is either Seifert fibered, toroidal or hyperbolic. For Seifert fibered spaces with two singular fibers, tight contact structures were completely classified, see \cite{Eliashberg, Etnyre:Lens, Giroux:classification, Honda:classification1}. For Seifert fibered spaces with three singular fibers, there have been several classification results, see \cite{EH:nonexistence, GS:classification, GLS:classification, Tosun, Wu} for example. For toroidal case, tight contact structures were classified on torus bundles over $S^1$ and $S^1$-bundles over a surface, see \cite{Giroux:3torus,Giroux:classification, Giroux:classification2, Honda:classification2, Kanda}. Honda, Kazez and Mati\'c \cite{HKM:decomposition} showed that there exist infinitely many tight contact structures up to isotopy on toroidal manifolds and Colin, Giroux and Honda \cite{CGH:finite} showed that a closed orientable irreducible 3-manifold supports infinitely many tight contact structures up to isotopy if and only if it is toroidal.

For hyperbolic manifolds, much less is known although most $3$-manifolds fall into this category \cite{Thurston:hyperbolic}. Honda, Kazez and Mati{\'c} \cite{HKM:hyperbolic} partially classified tight contact structures on hyperbolic surface bundles over $S^1$ when the contact structures have an {\it{extremal}} relative Euler class, meaning that the Euler class of the contact structure evaluated on a fiber surface is maximal. Later, Conway and the first author \cite{CM:figure-eight} classified tight contact structures on (most) surgeries on the figure eight knot in $S^3$, which contain infinitely many hyperbolic $3$-manifolds. 

In this paper we will consider hyperbolic L-spaces. Recall that an {\it L-space} is a rational homology sphere such that the order of the first singular homology group equals the free rank of its Heegaard Floer homology group. In general, it is hard to construct a universally tight contact structure on an L-space because an L-space does not admit a taut foliation, which can be perturbed into a universally tight contact structure. There are some non-hyperbolic L-spaces (\dfn{e.g.}~lens spaces) that support a universally tight contact structure while it is not known for hyperbolic L-spaces yet. Thus it is natural to ask whether there is a hyperbolic L-space that supports a universally tight contact structure, which can be considered as a contact topology version of the L-space conjecture. 

\begin{conjecture}\label{conj:lspace}
	A hyperbolic $3$-manifold is an L-space if and only if it does not support a universally tight contact structure.
\end{conjecture}

Notice that none of the hyperbolic manifolds mentioned above are an L-space, and they all admit a taut foliation so support universally tight contact structures. In this paper, we classify tight contact structures on an infinite family of hyperbolic L-spaces and show that most of them are virtually overtwisted.   

Let $M(n,r)$ be the result of $(n,r)$-surgery on the Whitehead link, see Figure~\ref{fig:Weeks1}. We will classify tight contact structures on $M(n,r)$ for surgery coefficients 
\begin{align*}
	&n \in \mathbb{Z}, n \geq 5,\\
	&r \in \mathcal{R}_+ \cup \mathbb{Q}_{-}, 
\end{align*}
where $\mathbb{Q}_-$ is the set of negative rational numbers and 
\begin{align*}
	&\mathcal{R}_+ = ([2,4) \cup [5,\infty)) \cap \mathbb{Q}.
\end{align*}
Recall that the Whitehead link is a hyperbolic L-space link. In particular, Liu \cite[Proposition 6.4]{Liu:lspace} showed that $M(r_1,r_2)$ is an L-space if and only if $r_1,r_2 > 0$. Moreover, there are only finitely many (26) isolated exceptional Dehn fillings of the Whitehead link \cite{Martelli:exceptional,Thurston:exceptional}. Thus our family of $M(n,r)$ contains infinite many hyperbolic L-spaces.

Before stating the results, we define two functions $\Phi(r)$ and $\Psi(r)$. For $0 < r \leq 1$, we write the negative continued fraction of $-\frac 1r$ as follows:
\[ 
	-\frac 1r = [r_0, \ldots, r_n] = r_0 - \frac1{r_1 - \frac1{\ddots \, - \frac1{r_n}}}, 
\]
where $r_0 \leq -1$, and $r_i \leq -2$ for $i = 1, \ldots, n$. Then define
\[ 
	\Phi(r) = \left|r_0(r_1+1)\cdots(r_n+1)\right|.
\]
We define $\Phi$ on all of $\Q$ by setting $\Phi(r+1) = \Phi(r)$.  Then we define $\Psi(r)$ for $r \in \Q$ by
\begin{align*}
	\Psi(r) = \begin{cases}0 & r=1 \\ \Phi(\frac1{1-r}) & r\neq 1 \end{cases}.
\end{align*}
Now we are ready to state the main theorem of the paper.

\begin{theorem}\label{thm:main}
	Let $n \geq 5$ be an integer. Then $M(n,r)$ supports
	\begin{align*}
		\begin{cases}
		\Psi(r) + 2\Phi(r)	& r \in \mathcal{R}_+\\
		\Psi(r)	& r \in \mathbb{Q}_- 
		\end{cases}
	\end{align*}
	tight contact structures up to isotopy, distinguished by their contact invariants in Heegaard Floer homology.
\end{theorem}

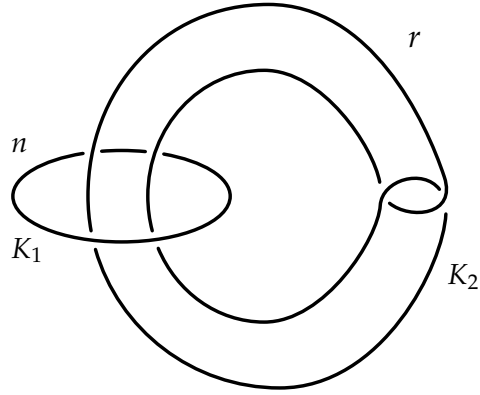
\begin{figure}[htbp]
\begin{center}
	\vspace{0.5cm}
	\begin{overpic}[tics=20]{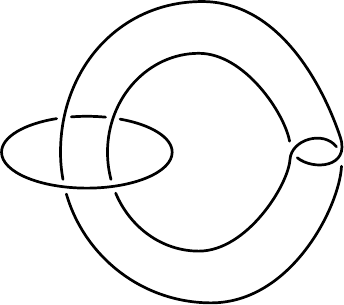}
		\put(0,90){$n$}
		\put(0,50){$K_1$}
		\put(150,130){\large $r$}
		\put(165,40){$K_2$}
	\end{overpic}
	\caption{$(n,r)$-surgery on the Whitehead link.}
	\label{fig:Weeks1}
\end{center}
\end{figure}

\begin{remark}
	For $n > 2$ and $r\geq1$ (\emph{resp.}~$r<1$), we can construct at least $\Psi(r) + 2\Phi(r)$ (\emph{resp.}~$\Psi(r)$) tight contact structures on $M(n,r)$ even if $n < 5$ or $r \notin \mathcal{R}_+$ (\emph{resp.}~$r \notin \mathbb{Q}_-$), see Section~\ref{sec:lower-bounds}. However, there could be more.
\end{remark}

We will prove Theorem~\ref{thm:main} by estimating an upper bound using convex surface decompositions, and realizing this upper bound via contact surgery. We construct tight contact structures in Section~\ref{sec:lower-bounds} by using contact surgery diagrams and distinguish them by using the contact invariant in Heegaard Floer homology. In Section~\ref{sec:upper-bounds}, we use a convex decomposition for a punctured genus one surface bundle over $S^1$ with a pseudo-Anosov monodromy, which was first introduced by Etnyre and Honda \cite{EH:knots} and developed by Conway and the first author \cite{CM:figure-eight} for the figure-eight knot complement. In this paper, we generalize this technique for any pseudo-Anosov monodromy $\phi$ with $-1 < c(\phi) <1$ where $c(\phi)$ is the fractional Dehn twist coefficient of $\phi$. We expect this technique will apply to various monodromies and lead to the classification of tight contact structures on a broad class of hyperbolic $3$-manifolds.      

Recall that $M(n,0)$ is a torus bundle over $S^1$ (see Section~\ref{subsec:decomp}). Since $\Phi(m) = 1$ and $\Psi(m) = 2$ for any integer $m > 2$ and $\Psi(m) = |m-1|$ for any integer $m < 0$, we have the following results for integer surgeries. 

\begin{corollary}
	Let $n \geq 5$ and $m \neq 1, 4$ be integers. Then $M(n,m)$ supports  
	\[
	\begin{cases}
		|m-1| & m < 0,\\
		\infty & m = 0,\\
		3 & m = 2,\\
		4	& m > 2,\, m \neq 4\\
	\end{cases}
	\]
	tight contact structures up to isotopy.
\end{corollary}

Recall that $M(5,\frac52)$ is the Weeks manifold, known to have the smallest volume among closed hyperbolic $3$-manifolds \cite{GMR:minvolume}. Since $\Phi(\frac52)=2$ and $\Psi(\frac52) = 3$, we obtain the following result.
\begin{corollary}
	The Weeks manifold supports seven tight contact structures up to isotopy.
\end{corollary}

In addition to classifying tight contact structures, we also determine whether these tight contact structures are Stein fillable.

\begin{theorem}\label{thm:fillability}
	Let $n \geq 5$ be an integer and $r \in \mathcal{R}_+ \cup \mathbb{Q}_-$. Then every tight contact structure on $M(n,r)$ is Stein fillable.
\end{theorem}

Lastly, we consider Conjecture~\ref{conj:lspace} for $M(n,r)$. We provide a lower bound on the number of tight contact structures on $M(n,r)$ that are virtually overtwisted.   

\begin{theorem}\label{thm:virtually-ot}
	Let $n \geq 5$ be an integer and $r \in \mathcal{R}_+ \cup \mathbb{Q}_-$. Then $M(n,r)$ supports at least 
	\begin{align*}
	\begin{cases}
		\Psi(r) + 2\Phi(r) - 6 & r \in \mathcal{R}_+\\
		\Psi(r) - 2 & r \in \mathbb{Q}_-
	\end{cases}
	\end{align*}
	virtually overtwisted contact structures up to isotopy.
\end{theorem}

\begin{remark}
  There are some cases (\emph{e.g.}~$r \ge 2$ integers) in which $\Psi(r)+2\Phi(r)$ is smaller than $6$. In these cases, it is uncertain whether any tight contact structure on $M(n,r)$ is universally tight or virtually overtwisted.
\end{remark}

\subsection*{Acknowledgments}
The authors thank James Conway and John Etnyre for their helpful comments on the first draft, Andy Wand and Ana Lecuona for their useful advices and Kenneth Baker for interesting comments on genus one fibered knots in lens spaces and their monodromies. The authors are also grateful to Lisa Piccirillo for a helpful discussion on surgery operations.

\section{Contact Topology Preliminaries} \label{sec:preliminaries}

We assume that the reader is familiar with basic $3$-dimensional contact topology \cite{Etnyre:contactlectures}, including Legendrian knots \cite{Etnyre:knots}, contact surgery, open book decomposition \cite{Etnyre:openbook}, convex surface theory \cite{Etnyre:convex,Honda:classification1}, and Heegaard Floer homology \cite{OS:HF2,OS:HF1}. In this section, we briefly review the results that we will frequently use.

\subsection{Convex surfaces} 
Let $(M,\xi)$ be a contact 3-manifold. Recall that a {\it contact vector field} is a vector field of which the flow preserves $\xi$. An orientable surface $\Sigma$ is called {\it convex} if there exists a contact vector field transverse to $\Sigma$. If $\Sigma$ has boundary, then we assume that $\bd\Sigma$ is Legendrian with negative twisting number with respect to the surface framing. According to \dfn{Giroux's flexibility theorem}, any embedded surface can be $C^\infty$-perturbed (rel boundary) into a convex surface. 

Given a convex surface $\Sigma$, the \dfn{dividing set} $\Gamma_{\Sigma}$ is a set of points on $\Sigma$ where the contact vector field $X$ is tangent to $\xi$. It can be shown that $\Gamma_{\Sigma}$ is an embedded multicurve on $\Sigma$. The dividing set divides $\Sigma$ into two regions: 
\[  
  \Sigma \setminus \Gamma_{\Sigma} = R_+ \cup R_-
\]
where $R_+ = \{p \in \Sigma : \alpha(X_p) > 0 \}$ and $R_- = \{p \in \Sigma : \alpha(X_p) < 0 \}$. If $\Sigma$ has Legendrian boundary, then $\tb(\bd\Sigma) = -\frac12 |\Gamma_{\Sigma} \cap \bd\Sigma|$ and $\rot(\bd\Sigma) = \chi(R_+) - \chi(R_-)$. 

We denote the twisting number of a Legendrian curve $\gamma$ with respect to a framing $F$ by $\tw(\gamma, F)$. We can easily compute the twisting number of a closed Legendrian curve on a convex surface.

\begin{theorem}[Honda \cite{Honda:classification1}]\label{thm:twisting} 
	Let $\gamma$ be an embedded Legendrian closed curve on a convex surface $\Sigma$ with a dividing set $\Gamma_{\Sigma}$. Then, 
	\[
		\tw(\gamma,\Sigma)=-\frac{|\gamma\cap\Gamma_{\Sigma}|}{2}
	\]
\end{theorem}

Let $\Sigma$ be a convex surface with a dividing set $\Gamma_{\Sigma}$. We call a properly embedded graph $G$ in $\Sigma$ {\it non-isolating} if $G$ intersect $\Gamma_{\Sigma}$ transversely and every component of $\Sigma\setminus G$ intersects $\Gamma_{\Sigma}$. 

\begin{theorem}[Legendrian realization principle, Honda \cite{Honda:classification1}]\label{thm:LRP} 
	If $G$ is a properly embedded non-isolating graph on a convex surface $\Sigma$, then there is an isotopic copy of $\Sigma$ relative to its boundary such that $G$ is Legendrian.
\end{theorem}

In particular, if $\Sigma\setminus G$ is connected, $G$ can always be realized as a Legendrian graph. This implies any non-separating curve on $\Sigma$ can be realized as a Legendrian curve.

Honda \cite{Honda:classification1} showed that if two convex surfaces intersect along a Legendrian curve $L$, the dividing curves interleave and we can obtain a new convex surface by \dfn{rounding the edge}, see Figure~\ref{fig:edgerounding}.

\begin{figure}[htbp]
	\begin{center}
		\vspace{0.5cm}
		\begin{overpic}[scale=0.9,tics=20]{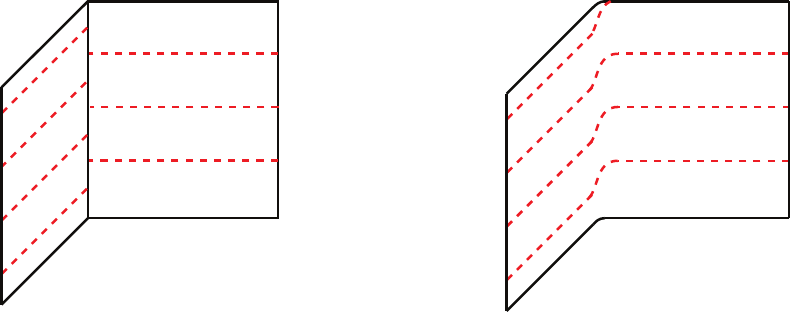}
		\end{overpic}
		\caption{Rounding an edge between two convex surfaces.}
		\label{fig:edgerounding}
	\end{center}
\end{figure}

We can modify a dividing set in a systematical way. Let $\Sigma$ be an oriented convex surface with a dividing set $\Gamma_{\Sigma}$ and $D$ be a disk with Legendrian boundary with $\tb(\partial D)=-1$. Suppose $\alpha=D\cap\Sigma$ intersects $\Gamma_{\Sigma}$ in three points $\{p,q,r\}$ and $\partial\alpha=\{p,r\}$ are elliptic singularity of $D$. By Giroux's flexibility theorem, we can perturb $D$ for $q$ to be a unique hyperbolic singularity. We call $D$ a {\it bypass} and the sign of the hyperbolic singularity is called the {\it sign} of a bypass.

Honda proved in \cite{Honda:classification1} that there is a one-sided neighborhood $\Sigma\times[0,1]$ of $\Sigma\cup D$ such that $\Sigma\times\{0,1\}$ is convex and the dividing set on $\Sigma\times\{1\}$ is modified as shown in Figure~\ref{fig:bypass}. We say $\Sigma\times\{1\}$ is obtained from $\Sigma\times\{0\}$ by a bypass attachment along $\alpha$. Note that a bypass can be on either side of a surface and it gives a different effect on the dividing set.

\begin{figure}[htbp]
	\begin{center}
		\vspace{0.5cm}
		\begin{overpic}[scale=0.7,tics=20]{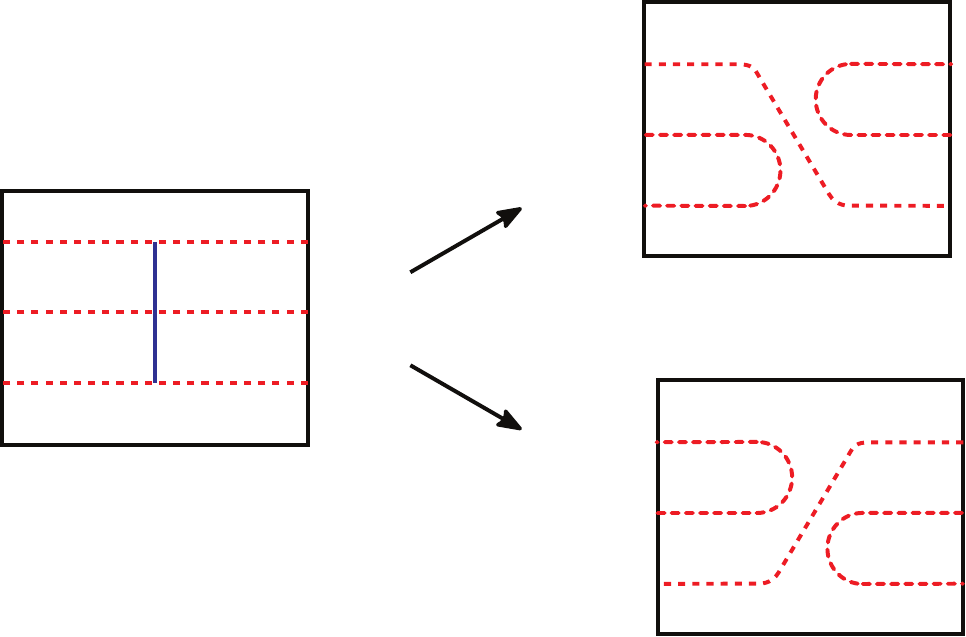}
		\put(150,150){front}
		\put(150,60){back}
		\end{overpic}
		\caption{Bypass attachments}
		\label{fig:bypass}
	\end{center}
\end{figure}

Let us investigate the effect of bypass attachment to a torus in detail. First, we introduce the Farey graph. We define the graph inductively. Let $\mathbb{D}$ be the Poincar\'e disk. Start with the top-most vertex labeled $0=\frac{0}{1}$ and the bottom-most vertex labeled $\infty=\frac{1}{0}$ and connect them by a geodesic. Given two vertices in the right half plane which are already labeled $\frac{a}{b}$ and $\frac{c}{d}$, choose a vertex on $\bd\mathbb{D}$ in the middle of the vertices and label it as $\frac{a+c}{b+d}$. Then connect it with the other two vertices by geodesics. For the left half plane, we treat $\infty$ as $-\infty=\frac{-1}{0}$. See Figure~\ref{fig:Farey}.

\begin{figure}[htbp]
	\begin{center}
		\vspace{1cm}
		\begin{overpic}[scale=0.5,tics=20]{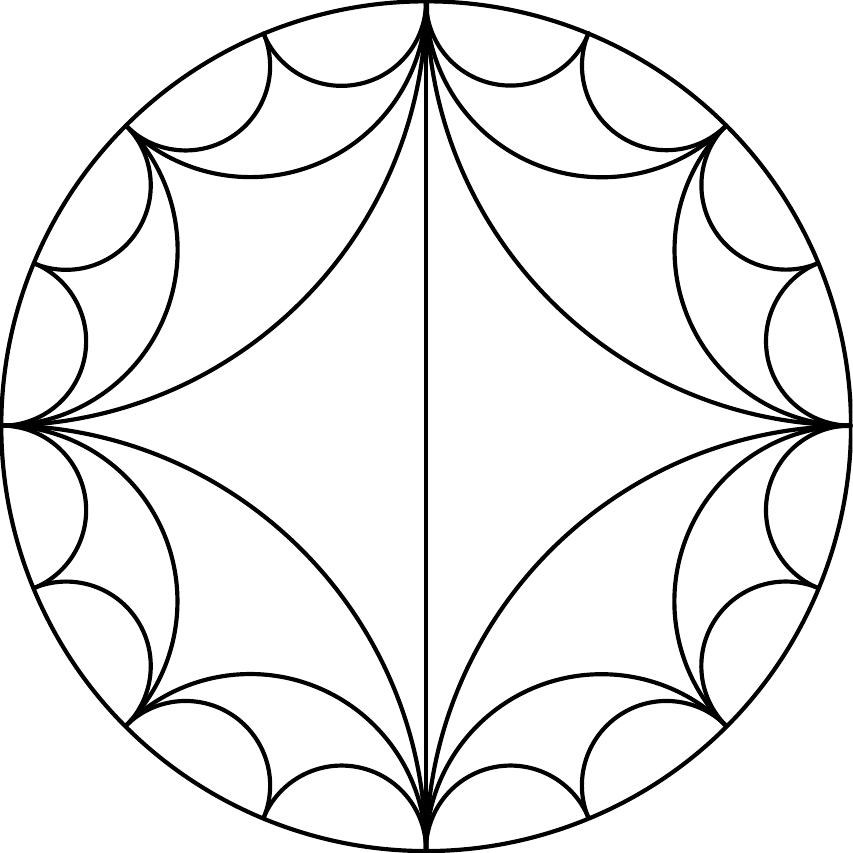}
		\put(100,213){$0$}		
		\put(99,-12){$\infty$}		
		\put(210,100){$\frac{1}{1}$}		
		\put(-20,100){$-\frac{1}{1}$}		
		\put(180,180){$\frac{1}{2}$}
		\put(180,20){$\frac{2}{1}$}
		\put(15,20){$-\frac{2}{1}$}
		\put(15,180){$-\frac{1}{2}$}
		\put(140,205){$\frac{1}{3}$}
		\put(200,150){$\frac{2}{3}$}
		\put(200,60){$\frac{3}{2}$}
		\put(140,-5){$\frac{3}{1}$}
		\put(50,-5){$-\frac{3}{1}$}
		\put(-10,60){$-\frac{3}{2}$}
		\put(-10,145){$-\frac{2}{3}$}
		\put(50,205){$-\frac{1}{3}$}	
		\end{overpic}
		\vspace{0.5cm}
		\caption{The Farey graph}
		\label{fig:Farey}
	\end{center}
\end{figure}
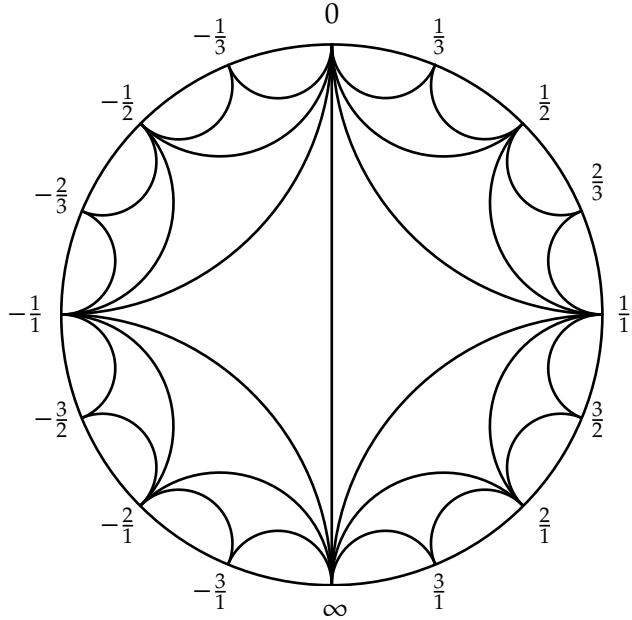

Now consider a convex torus $T^2$ with two parallel dividing curves and a linear characteristic foliation. We call a leaf of this characteristic foliation a \emph{ruling curve}. Choose a homology basis for $T^2$ as $(\vects{1}{0},\vects{0}{1})$. Denote the slope of curves parallel to $\vects{p}{q}$ by $\frac{p}{q}$.

\begin{theorem}[Honda \cite{Honda:classification1}]\label{thm:bypass}
	Let $s,r$ be the dividing slope and the ruling slope of $T^2$ respectively. After a bypass attachment along a ruling curve, we obtain a new convex torus with two dividing curves with slope $s'$, where $s'$ is the vertex on the Farey graph anticlockwise of r and clockwise of s. In addition, $s'$ is closest to $r$ with an edge to $s$.	 
\end{theorem}

\begin{remark}\label{rmk:convention}
	In Theorem~\ref{thm:bypass}, we use the slope convention $\frac pq = \vects{p}{q}$ since $\frac{\text{meridian}}{\text{longitude}}$ is our slope convention for a solid torus. If we use the slope convention $\frac qp = \vects{p}{q}$, then we should reverse the words ``clockwise'' and ``anticlockwise''. Also, Theorem~\ref{thm:bypass} assumes that we attach a bypass from the front. If we attach a bypass from the back, then we should reverse the words ``clockwise'' and ``anticlockwise''.
\end{remark}

We need to guarantee the existence of bypasses to utilize them. There are several ways to find bypasses.

\begin{theorem}[Imbalance principle, Honda \cite{Honda:classification1}]\label{thm:Imbalance} Let $\Sigma$ and $A=S^1\times[0,1]$ be convex surfaces with Legendrian boundary. Suppose $\Sigma\cap A=S^1\times\{0\}$ and $\tw(S^1\times\{0\})<\tw(S^1\times\{1\})\leq0$. Then there is a bypass for $\Sigma$ along $S^1\times\{0\}$.
\end{theorem}

\begin{theorem}[Disk imbalance, Honda \cite{Honda:classification1}]\label{thm:DiskImbalance}
	Let $\Sigma$ be a convex surface and $D$ be a convex disk with a Legendrian boundary. Suppose $\Sigma\cap D=\bd D$. If $\tb(\bd D)<-1$ then there is a bypass for $\Sigma$ along $\bd D$.
\end{theorem}

\begin{theorem}[Bypass rotation, Honda--Kazez--Mati\'c \cite{HKM:pinwheels}] \label{thm:bypass-rotation} Suppose that there is a bypass for $\Sigma$ from the front along an attaching arc $\gamma$. If $\gamma'$ is a Legendrian arc as in Figure~\ref{fig:bypass-rotation}, then there exists a bypass for $\Sigma$ from the front along $\gamma'$.
\end{theorem}

\begin{theorem}[Bypass slide, Honda \cite{Honda:classification1}] \label{thm:bypass-sliding}
	Suppose that there is a bypass for $\Sigma$ from the front along an attaching arc $\gamma$. If $\gamma'$ is a Legendrian arc that is isotopic to $\gamma$ relative to $\Gamma_\Sigma$, then there is a bypass for $\Sigma$ from the front along $\gamma'$.
\end{theorem}

\begin{figure}[htbp]
\begin{center}
	\begin{overpic}[scale=1,tics=20]{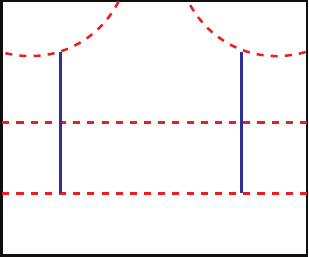}
	\put(120,78){$\gamma$}
	\put(32,78){$\gamma'$}
	\end{overpic}
	\caption{The attaching arc for a bypass rotation.}
	\label{fig:bypass-rotation}
\end{center}
\end{figure}

\subsection{Contact structures on \texorpdfstring{$T^2\times I$}{T2 X I} and solid tori}
A tight contact structure $\xi$ on $T^2\times I$ is called a {\it basic slice} if 

\begin{enumerate}
	\item $\bd (T^2 \times I) = T_0\cup T_1$ is two standard convex tori with dividing slopes $s_0$ and $s_1$, respectively,
	\item $s_0$ and $s_1$ are connected by an edge in the Farey graph, and
	\item the slope of the dividing curves on any convex torus $T$ parallel to the boundary is clockwise of $s_0$ and anticlockwise of $s_1$ in the Farey graph. This condition is called {\it minimal twisting}.
\end{enumerate}

Honda \cite{Honda:classification1} studied various properties of basic slices. Here, we review the classification of basic slices up to isotopy. 

\begin{theorem}[Honda \cite{Honda:classification1}]
	There are exactly two basic slices up to isotopy fixing the characteristic foliation on the boundary. These two tight contact structures are distinguished by their relative Euler class. They are called positive and negative basic slices and denoted by $B_+(s_0,s_1)$ and $B_-(s_0,s_1)$, respectively.
\end{theorem}

We now review Giroux torsion. Let $k \in \frac12 \N$. Consider the contact structure $\xi^k=\ker(\sin (2\pi kz) dx + \cos( 2\pi kz) dy)$ on $T^2\times [0,1]$, where $(x,y)$ are the coordinates on $T^2$ and $z$ is the coordinate on $[0,1]$. We call $\xi^k$ a \dfn{Giroux $k$-torsion layer} and if we have a contact structure $(M,\xi)$ into which $(T^2\times[0,1],\xi^k)$ embeds, we say $(M,\xi)$ contains \dfn{Giroux $k$-torsion}. We will use the phrase $(M,\xi)$ has \dfn{exactly Giroux $k$-torsion} to the situation where one can embed  $(T^2\times[0,1],\xi^k)$ into $(M,\xi)$ but not $(T^2\times[0,1],\xi^{k+\scriptscriptstyle\frac 12})$. If $k=1/2$, then we call $\xi^k$ a \dfn{half Giroux torsion}. We say $(M,\xi)$ has \dfn{no (half) Giroux torsion} if $(T^2\times[0,1],\xi^k)$ does not embed in $(M,\xi)$ for any $k \in \frac12\N$. 

Now we consider $(T^2\times \mathbb{R}, \sin (2\pi kz) dx + \cos (2\pi kz) dy)$ and perturb $T^2\times\{0\}$ and $T^2\times \{1\}$ so that they become convex with two dividing curves of slope $0$. Let $\xi_c^k$ be the resulting contact structure on $T^2\times [0,1]$. We call this $\xi_c^k$ a \dfn{convex Giroux $k$-torsion layer}. Notice that inside of $(T^2\times[0,1], \xi_c^k)$ there is a basic slice with one boundary component agreeing with $T^2\times\{0\}$ and having boundary slope $0$ and $\infty$. This will either be a positive or a negative basic slice. By reversing the coorientation of $\xi_c^k$, if necessary, we can assume it is positive. 

\begin{theorem}\label{twolayers}
  For $k \in \frac12\N$, there are exactly two convex Giroux $k$-torsions $\pm \xi_c^k$, up to isotopy, on $T^2\times [0,1]$ with convex boundary having two dividing curves, both of slope $r$. The two contact structures are contactomorphic. 
\end{theorem}

\begin{remark}
	In this paper, since we only consider contact manifolds with convex boundary, we just use the term \dfn{Giroux $k$-torsion} for convex Giroux $k$-torsion. However, readers should notice that the original Giroux torsion has pre-Lagrangian torus boundary.  
\end{remark}

Next, we will review the classification of tight contact structures on a solid torus. Consider a solid torus with a fixed longitude. Here, our slope convention is $\frac{\text{meridian}}{\text{longitude}}$. We denote by $N(s)$, a solid torus with convex boundary having two dividing curves of slope $s$.

\begin{theorem}[Giroux \cite{Giroux:classification}, Honda \cite{Honda:classification1}]\label{thm:solid-torus}
	Let $(p,q)$ be a pair relatively prime integers satisfying $q>-p\geq1$ and
	\[
		\frac{q}{p} = [r_0, \cdots, r_k]
	\]
	for $r_i\leq-2$. Then there exist 
	\[
		|(r_0+1)\ldots(r_{k-1}+1)r_k|  
	\]
	tight contact structures on $N(\frac pq)$ up to isotopy fixing the characteristic foliation on the boundary. 
\end{theorem} 

Now we consider a solid torus with a different meridional slope. Take a core of $N(s)$ and perform an $r$-surgery on it. Then we obtain a solid torus with dividing slope $s$ and meridional slope $r$. Denote it by $N_r(s)$. From Theorem~\ref{thm:solid-torus}, we can deduce the following result.

\begin{proposition}[Conway--Min \cite{CM:figure-eight}]\label{prop:solid-torus}
	Let $(p,q)$ be a pair relatively prime integers. Then there exist $\Phi(\frac pq)$ tight contact structures on $N_{p/q}(\infty)$ up to isotopy fixing the characteristic foliation on the boundary.
\end{proposition}

\subsection{Contact surgery}\label{sec:contactsurgery} 
Let $L$ be a Legendrian knot in a contact $3$-manifold $(M,\xi)$ and $N$ a standard neighborhood of $L$. {\it Contact $(\frac{p}{q})$-surgery} on $L$ is defined as follows: let $(\mu,\lambda)$ be a meridian and contact framing of $L$, respectively. $(M_{(p/q)}(L), \xi_{(p/q)})$ is obtained by cutting $N$ from $M$ and re-gluing it via a diffeomorphism of $\bd N$ sending $\mu$ to $p\mu+q\lambda$. Then, extend the contact structure $(M\setminus N,\xi)$ to the entire $M_{(p/q)}(L)$. In general, the extension is not unique since there could be distinct tight contact structures on $N$ with a given boundary condition. If we only focus on the tight contact structures on $N$, Theorem~\ref{thm:solid-torus} provides the number of all possible extensions. 

\begin{theorem}[Ding and Geiges \cite{DG:surgery}]\label{contact surgery}
	Let $L$ be a Legendrian knot in $(M,\xi)$ and $p,q$ be relatively prime integers satisfying $\frac{p}{q}<0$. The number of contact structure induced by contact $(\frac{p}{q})$-surgery on $L$ is 
	\[
		|(r_0+1)\cdots(r_n+1)|
	\] 
	where 
	\[
		\frac{p}{q}=r_0+1-\frac{1}{r_1-\frac{1}{r_2\ldots-\frac{1}{r_n}}} = [r_0+1, r_1, \cdots,r_n]
	\] 
	for $r_i\leq-2$.
\end{theorem}

If $\frac{p}{q} > 0$, use $\frac{p}{q-kp}$ instead where $k$ is a positive integer such that $q-kp<0$. 

In \cite{DGS:surgery}, Ding, Geiges and Stipsicz exhibited an algorithm that converts any contact surgery diagram into a $(\pm1)$-surgery diagram.

\begin{itemize}
	\item contact $(\frac{p}{q})$-surgery with $\frac{p}{q}<0$:
	\begin{enumerate}
		\item Stabilize $L$ $|r_0+2|$ times. Let this be $L_0$.
		\item For $i=1,\ldots,n$, let $L_i$ be the Legendrian push-off of $L_{i-1}$ and stabilize it $|r_i+2|$ times.
		\item Then a contact $(\frac{p}{q})$-surgery on $L$ corresponds to a contact $(-1)$-surgeries on a link $(L_0,\ldots,L_n)$.
	\end{enumerate}
	\item contact $(\frac{p}{q})$-surgery with $\frac{p}{q}>0$:
	\begin{enumerate}
		\item Choose a positive integer $k$ such that $q-kp<0$. Let $r'=\frac{p}{q-kp}$.
		\item Let $L_1,\ldots,L_k$ be k successive Legendrian push-offs of $L$.
		\item Then a contact $(\frac{p}{q})$-surgery on $L$ corresponds to $(+1)$-surgeries on $L,L_1,\ldots,L_{k-1}$ and a contact $(r')$-surgery on $L_k$.
	\end{enumerate}
\end{itemize}

Sometimes, a contact surgery produces an overtwisted contact structure.

\begin{proposition}\label{prop:OTsurgery}
	A contact $(r)$-surgery on a Legendrian knot $L$ results in an overtwisted contact structure if
	\begin{itemize}
		\item $r=0$ and $L$ is any Legendrian knot,
		\item $0< r < 1$ and $L$ is a Legendrian unknot with $\tb = -1$.
	\end{itemize} 
\end{proposition}

\subsection{Contact invariants in Heegaard Floer homology}
In this subsection, we review some useful properties of contact invariants in Heegaard Floer homology.

\begin{theorem}[Ozsv{\'a}th and Szab{\'o} \cite{OS:contact}, Ghiggini \cite{Ghiggini:HF}]\label{thm:contact-invariant}
	The Heegaard Floer contact invariant $c(\xi) \in \HFhat(-M)$ of $(M,\xi)$ satisfies the following properties.
	\begin{itemize}
		\item If $(M,\xi)$ is overtwisted, then $c(\xi)=0$.
		\item If $(M,\xi)$ is strongly fillable, then $c(\xi)\neq0$.
		\item Let $L$ be a Legendrian knot in $(M,\xi)$ and $(M_L,\xi_L)$ the result of Legendrian surgery on $L$. If $c(\xi) \neq 0$, then $c(\xi_L) \neq 0$. 
	\end{itemize}
\end{theorem}

There also have been several studies when a positive contact surgery preserves non-vanishing contact invariants, see \cite{Golla:surgery,LS:trefoil,LS:HF2,MT:surgery}. In this paper, we only need the result for the right-handed trefoil.

\begin{theorem}[Lisca and Stipsicz \cite{LS:trefoil}]\label{thm:rht}
	Let $L$ be a Legendrian right-handed trefoil in $(S^3,\xi_{std})$ with $tb(L)=1$. For any $r\in\mathbb{Q}\setminus\{0\}$, contact $(r)$-surgery on $L$ produces a tight contact structure with non-vanishing contact invariant.
\end{theorem}

\section{The lower bounds}\label{sec:lower-bounds}

In this section, we will construct $2\Phi(r) + \Psi(r)$ isotopy classes of tight contact structures on $M(n,r)$ for $n > 2$ and $r \geq 1$ and $\Psi(r)$ isotopy classes of tight contact structures for $n>2$ and $r < 1$ using contact surgery diagrams. We first construct contact structures counted by $\Psi$ and then consider contact structures counted by $\Phi$. After that, we will show that the contact structures counted by $\Psi$ are distinct from the ones counted by $\Phi$.

\subsection{Tight contact structures counted by \texorpdfstring{$\Psi$}{Psi}.}
Consider the surgery diagram for $M(n,r)$ shown in Figure~\ref{fig:Weeks1}. A left-handed Rolfsen twist on $K_1$ turns the diagram into the one in Figure~\ref{fig:Rolfsen}. We then perform a sequence of inverse slam--dunk moves to turn $K_1$ into a chain of unknots. Since we have 
\begin{align*}
	-\frac{n+1}{n}= [\overbrace{-2,\cdots, -2}^n], 
\end{align*}
there are $n$ components in the chain of the unknots and the surgery coefficients are all $-2$, see Figure~\ref{fig:slamdunk1}. We can turn this diagram  into a contact surgery diagram as shown in Figure~\ref{fig:legendrian1}. Applying the algorithm in Section~\ref{sec:contactsurgery}, we can convert the diagram into a $(\pm1)$-contact surgery diagram. Notice that there are choices of stabilizations during the conversion.

\begin{figure}[htbp]
\begin{center}
	\vspace{0.5cm}
	\begin{overpic}[tics=20]{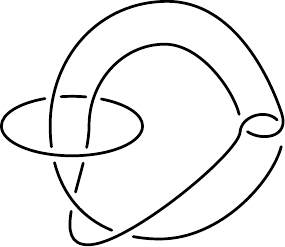}
		\put(-15,75){$-\frac{n}{n-1}$}
		\put(125,100){$r$}
	\end{overpic}
	\caption{A surgery diagram for $M(n,r)$.}
	\label{fig:Rolfsen}
\end{center}
\end{figure}
	
\begin{figure}[htbp]
\begin{center}
	\vspace{0.5cm}
	\begin{overpic}[tics=20]{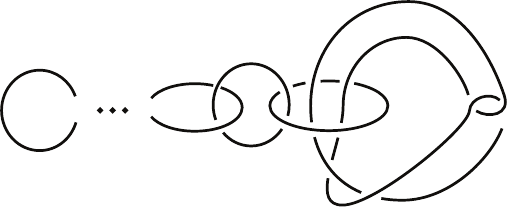}
		\put(-3,68){\footnotesize $-2$}
		\put(74,63){\footnotesize $-2$}
		\put(103,70){\footnotesize $-2$}
		\put(183,60){\footnotesize $-2$}
		\put(240,80){$r$}
	\end{overpic}
	\caption{The result of inverse slam--dunk moves.}
	\label{fig:slamdunk1}
\end{center}
\end{figure}
	
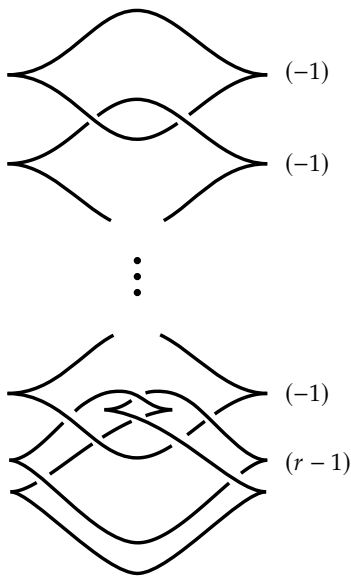
\begin{figure}[htbp]
\begin{center}
	\begin{overpic}[tics=20]{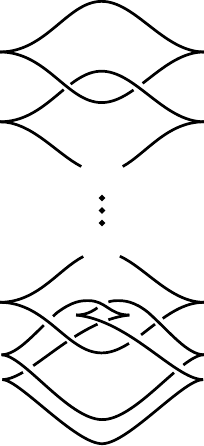}
		\put(105,188){\footnotesize $(-1)$}
		\put(105,153){\footnotesize $(-1)$}
		\put(105,66){\footnotesize $(-1)$}
		\put(105,40){\footnotesize $(r-1)$}
	\end{overpic}
	\caption{A contact surgery diagram for tight contact structures on $M(n,r)$ counted by $\Psi$.}
	\label{fig:legendrian1}
\end{center}
\end{figure}

\begin{proposition}\label{prop:surgerytightPsi}
	The contact surgery diagram in Figure~\ref{fig:legendrian1} induces a tight contact structure for any choice of stabilizations for any $r \neq 1$.
\end{proposition}

\begin{proof}
	We first consider the case $r > 1$. Fix a choice of stabilizations. Since $r > 1$, the contact surgery coefficient on the trefoil component in Figure~\ref{fig:legendrian1} is positive. Thus there exists a single link component $L$ having contact surgery coefficient $(+1)$ after the conversion, and all other components have the contact surgery coefficient $(-1)$. Let $(S^3_2(L), \xi_L)$ be the result of a contact $(+1)$-surgery on $L$. $L$ is a Legendrian right-handed trefoil with $\tb(L)=1$, hence the Heegaard Floer invariant of $\xi_L$ is non-vanishing by Theorem~\ref{thm:rht}. Since the contact structure on $M(n,r)$ is constructed from $(S^3_2(L), \xi_L)$ by Legendrian surgeries on the other link components, the Heegaard Floer invariant of the resulting contact structure is non-vanishing by Theorem~\ref{thm:contact-invariant}. Thus the contact structure on $M(n,r)$ is tight.
	
	Now we consider the case $r < 1$. Fix a choice of stabilizations. Since $r < 1$, the contact surgery coefficient on the trefoil component in Figure~\ref{fig:legendrian1} is negative. Thus this contact surgery diagram represents a Stein fillable contact structure and hence it is tight. 
\end{proof}

\begin{proposition}\label{prop:surgeryPsi}
	The contact surgery diagram in Figure~\ref{fig:legendrian1} induces $\Psi(r)$ pairwise non-isotopic tight contact structures for any $r \in \mathbb{Q}$, distinguished by their Heegaard Floer invariants.
\end{proposition}

\begin{proof}
	We first consider the case $r > 1$. We will use the same notations used in Proposition~\ref{prop:surgerytightPsi}. As we discussed in the proof of Proposition~\ref{prop:surgerytightPsi}, there exists a single link component $L$ with the contact surgery coefficient $(+1)$ after the conversion, and all other components have the contact surgery coefficient $(-1)$. Thus there exists a Stein cobordism $W$ from $S^3_2(L)$ to $M(n,r)$ induced from the contact surgery diagram. We claim that different choices of stabilizations give rise to non-isomorphic Stein structures and by \cite[Theorem~2]{Plamenevskaya:cobordism} the induced contact structures on $M(n,r)$ are also non-isotopic, distinguished by their Heegaard Floer invariants. To prove the claim, fix two different choices of stabilizations. Suppose $L_1,...,L_m$ are the push-offs of $L$ after the conversion. Since each $L_i$ is rationally null-homologous in $S^3_2(L)$, we can calculate its rational rotation number in $(S^3_2(L),\xi_L)$. Recall that a stabilization behaves precisely the same way as in the null-homologous case, i.e. $\rot_\mathbb{Q}(S_\pm(L)) = \rot_\mathbb{Q}(L) \pm 1$ (cf.~\cite[Lemma~1.3]{BE:rational}). 
    
  Since we chose different sets of stabilizations, there exists at least one index $i$ such that the values of $\rot_\mathbb{Q}(L_i)$ are different. Now we compare the first Chern class of the Stein structure $J$ induced by the surgery diagrams. Using the formula from \cite[Lemma~4.1]{LW:rational}, we obtain 
	\[
		\langle c_1(J), [\Sigma_i \cup 2 \cdot C_i] \rangle = 2 \cdot \rot_\mathbb{Q}(L_i),
	\]
	where $\Sigma_i$ is a rational Seifert surface of $L_i$, $C_i$ is the core of the Weinstein $2$-handle attached along $L_i$. Thus, the two different values of $\rot_\mathbb{Q}(L_i)$ lead to non-isomorphic Stein structures, proving the claim. 

	By Proposition~\ref{prop:surgerytightPsi}, all contact structures are tight for any choices of stabilizations. Hence we are only left to show that there exist $\Psi(r)$ different choices of stabilizations for the contact surgery diagram in Figure~\ref{fig:legendrian1}. According to the algorithm in Section~\ref{sec:contactsurgery}, if $r > 1$, a contact $(r-1)$-surgery on $L$ is equivalent to a contact $(+1)$-surgery on $L$ and $(-\frac{r-1}{r-2})$-surgery on its push-off. Consider the negative continued fraction expansion $-\frac{r-1}{r-2} = [r_0,r_1,...,r_k]$. Then the number of stabilization choices for a contact $(-\frac{r-1}{r-2})$-surgery is $|r_0(r_1+1) \cdots (r_k+1)|$, which is $\Phi(\frac{r-2}{r-1}) = \Phi(\frac{1}{1-r}) = \Psi(r)$. 
	
	If $r =1$, we have $\Psi(1)=0$ and the contact surgery diagram produces an overtwisted contact structure by Proposition~\ref{prop:OTsurgery}.

	Lastly, we consider the case $r < 1$. As discussed in Proposition~\ref{prop:surgerytightPsi}, since all contact surgery coefficients are negative, the surgery diagram represents a Stein manifold $(X,J)$ and the first Chern class of the Stein structure will evaluate to $\rot(L_i)$ on the set of generators of $H_2(X;Z)$ given by the cores of $2$-handles and the Seifert surfaces for the attaching spheres, see \cite{Gompf:stein}. Since we chose different sets of stabilizations, there exists at least one index $i$ such that the values of $\rot(L_i)$ are different. Thus, the two different values of $\rot(L_i)$ lead to non-isomorphic Stein structures. Now we remain to show that there exist $\Psi(r)$ different choices of stabilizations for the contact surgery diagram in Figure~\ref{fig:legendrian1}. Consider the negative continued fraction expansion $r-1 = [r_0,r_1,...,r_k]$. According to the algorithm in Section~\ref{sec:contactsurgery}, the number of stabilization choices for a contact $(r-1)$-surgery is $|r_0(r_1+1) \cdots (r_k+1)|$, which is $\Phi(\frac{1}{1-r}) = \Psi(r)$. 
\end{proof}

\subsection{Tight contact structures counted by \texorpdfstring{$\Phi$}{Phi}.} 
Since the Whitehead link is symmetric, we can switch the link components as shown in the first drawing of Figure~\ref{fig:Weeks2}. Perform a left-handed Rolfsen twist on $K_2$, and we obtain a surgery diagram as shown in the right drawing of Figure~\ref{fig:Weeks2}. After realizing this link as a Legendrian link, we obtain a contact surgery diagram shown in Figure~\ref{fig:legendrian2}. Applying the algorithm from Section~\ref{sec:contactsurgery}, we can convert the diagram into a $(\pm1)$-contact surgery diagram. Again, there are choices of stabilizations during the conversion.

The proofs of Proposition~\ref{prop:surgerytightPhi} and \ref{prop:surgeryPhi} are essentially identical to the proofs of Proposition~\ref{prop:surgerytightPsi} and \ref{prop:surgeryPsi}, respectively. We leave the first one as an exercise for the readers.

\begin{figure}[htbp]
\begin{center}
	\begin{overpic}[tics=20]{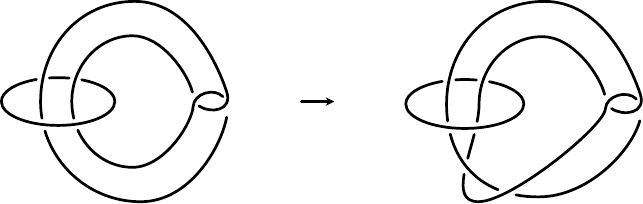}
		\put(0,60){\footnotesize $r$}
		\put(110,70){\footnotesize $n$}
		\put(180,65){$-\frac{r}{r-1}$}
		\put(310,70){\footnotesize $n$}
	\end{overpic}
	\caption{Another surgery diagram for $M(n,r)$.}
	\label{fig:Weeks2}
\end{center}
\end{figure}

\begin{figure}[htbp]
\begin{center}
	\begin{overpic}[tics=20]{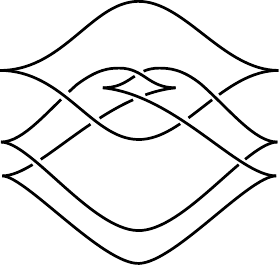}
		\put(140,90){$(-\frac{1}{r-1})$}
		\put(140,45){\footnotesize $(n-1)$}	
	\end{overpic}
	\caption{A contact surgery diagram for tight contact structures on $M(n,r)$ counted by $\Phi$.}
	\label{fig:legendrian2}
\end{center}
\end{figure}

\begin{proposition}\label{prop:surgerytightPhi}
	The contact surgery diagram in Figure~\ref{fig:legendrian2} induces a tight contact structure for any choices of stabilizations for $r\geq 1$. \qed
\end{proposition}

\begin{remark}
	By Proposition~\ref{prop:OTsurgery}, the surgery diagram in Figure~\ref{fig:legendrian2} results in overtwisted contact structures if $r < 0$.
\end{remark}

\begin{proposition} \label{prop:surgeryPhi}
	The contact surgery diagram in Figure~\ref{fig:legendrian2} induces $2\Phi(r)$ pairwise non-isotopic tight contact structures for $r\geq1$, distinguished by their Heegaard Floer invariants. 
\end{proposition}

\begin{proof}
	Similar to the proof of Proposition~\ref{prop:surgerytightPsi}, there exists a single link component $L$ with the contact surgery coefficient $(+1)$ after the conversion, and all other components have the contact surgery coefficient $(-1)$. Thus there exists a Stein cobordism $W$ from $S^3_2(L)$ to $M(r)$ induced from the surgery diagram. As in the proof of Proposition~\ref{prop:surgeryPsi}, different choices of stabilizations give rise to non-isomorphic Stein structures on $W$ and hence to non-isotopic contact structures on $M(n,r)$.   
	
	By Proposition~\ref{prop:surgerytightPhi}, all contact structures are tight for any choices of stabilizations. Thus we are only left to show that there exist $2\Phi(r)$ different choices of stabilizations for the contact surgery diagram in Figure~\ref{fig:legendrian2}. According to the algorithm in Section~\ref{sec:contactsurgery}, there exists two choices of stabilizations for a contact $(n-1)$-surgery. For a contact $(-\frac{1}{r-1})$-surgery, consider the negative continued fraction expansion of $-\frac{1}{r-1} = [r_0,r_1,...,r_k]$. Then the number of stabilization choices for a contact $(-\frac{1}{r-1})$-surgery is $|r_0(r_1+1) \cdots (r_k+1)|$, which is $\Phi({r-1}) = \Phi(r)$.
\end{proof}
      
\subsection{Distinguishing contact structures counted by \texorpdfstring{$\Psi$}{Psi} and \texorpdfstring{$\Phi$}{Phi}} 
So far, we have constructed tight contact structures on $M(n,r)$ using contact surgery diagrams. However, since we have used different surgery diagrams for the ones counted by $\Psi$ and $\Phi$, respectively, we constructed two distinct manifolds diffeomorphic to $M(n,r)$. To compare the isotopy classes of the contact structures on each manifold, we push them to a fixed manifold and determine their first Chern classes. 

\begin{figure}[htbp]
	\begin{center}
		\begin{overpic}[tics=20]{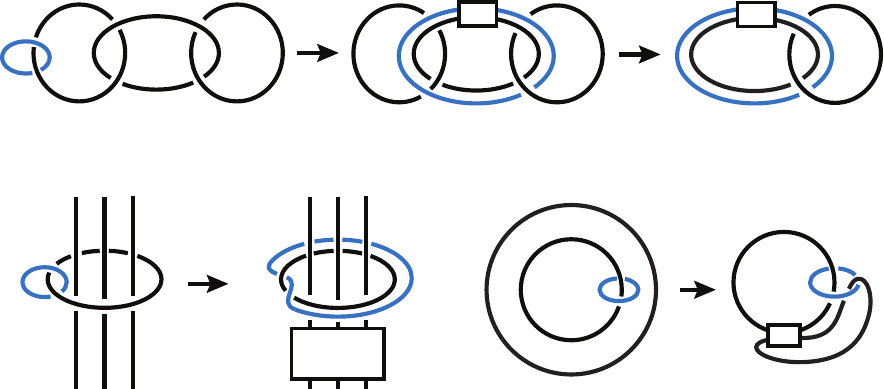}
			\put(30,188){$m$}
			\put(68,182){$n$}	
			\put(125,186){$k$}	

			\put(185,188){$m$}
			\put(227,149){$n$}	
			\put(227,178){$n$}	
			\put(282,186){$k$}	

			\put(348,152){$n-\frac1m$}	
			\put(361,178){$n$}	
			\put(416,186){$k$}	
			\put(158,14){ $1$}	

			\put(251,88){$m$}	
			\put(406,10){$m+n$}	
			\put(260,74){$n$}	
			\put(356,74){$n$}	
			\put(375,23){\small $n$}	
		\end{overpic}
		\caption{Various surgery operations.}
		\label{fig:surgeryOperations}
	\end{center}
\end{figure}
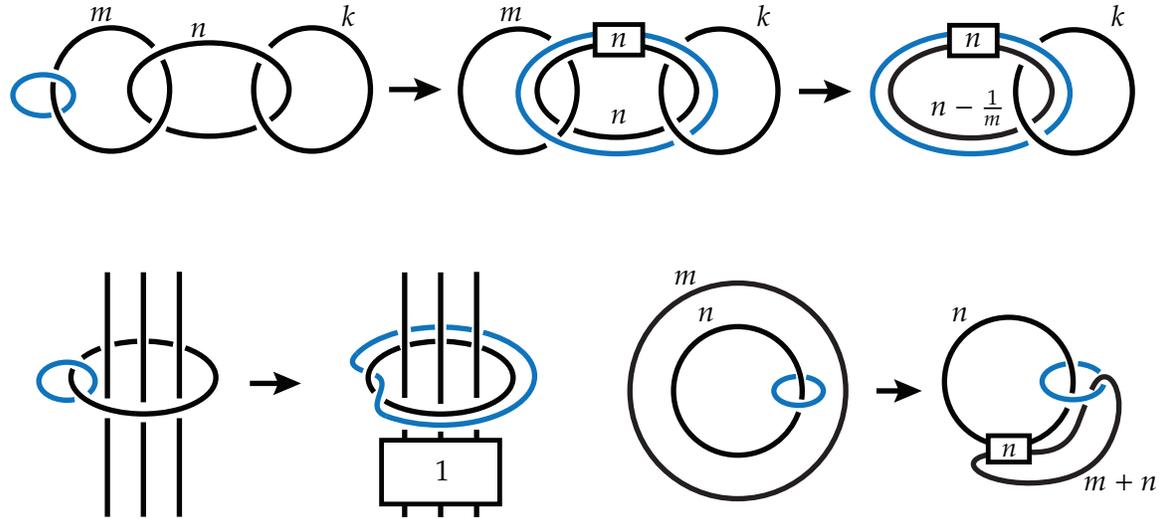

\begin{proposition}\label{prop:lowerbounds}
	For $n>2$ and $r \geq 1$, $M(n,r)$ supports at least $\Psi(r) + 2\Phi(r)$ pairwise non-isotopic tight contact structures, distinguished by their Heegaard Floer invariants.
\end{proposition}

\begin{proof}
	In the previous subsections, we denoted by $M(n,r)$ any 3-manifolds constructed from the surgery diagrams in Figure~\ref{fig:Weeks1},~\ref{fig:slamdunk1},~\ref{fig:legendrian1},~\ref{fig:Weeks2}, and~\ref{fig:legendrian2}, but here we will distinguish them. Let $M(n,r)$ be the 3-manifold constructed from the surgery diagram shown in Figure~\ref{fig:Weeks1}, and $M_1$ a manifold constructed from the $(\pm1)$--contact surgery diagram which is converted from Figure~\ref{fig:legendrian1} using the algorithm in Section~\ref{sec:contactsurgery}, and $M_2$ a manifold constructed from the $(\pm1)$--contact surgery diagram converted from Figure~\ref{fig:legendrian2}. It is routine to check that we can return back from these $(\pm1)$--contact surgery diagrams to the surgery diagram in Figure~\ref{fig:Weeks1} by using isotopies (handle slides), slam--dunk moves and Rolfsen twists. Since each of these operations induces a diffeomorphism, we obtain diffeomorphisms $f^i\colon M_i \to M(n,r)$ for $i=1,2$.  

	Pick contact structures $(M_1, \xi_1)$ and $(M_2, \xi_2)$ counted by $\Psi$ and $\Phi$, respectively. We claim that the contact structures $f^1_*(\xi_1)$ and $f^2_*(\xi_2)$ on $M(n,r)$ are not isotopic. Together with Proposition~\ref{prop:surgeryPsi} and \ref{prop:surgeryPhi}, the claim completes the proof.
	
	We will show the claim by calculating the first Chern classes of the contact structures. We start with calculating $H_1(M(n,r))$. Let $\mu$ and $\nu$ be the meridians of $K_1$ and $K_2$ in Figure~\ref{fig:Weeks1}, respectively. Since the linking number between the components of the Whitehead link is 0, we have 
	\[
		H_1(M(n,r)) = \langle \mu, \nu \mid n\mu =0,  p \nu= 0 \rangle = \Z_n \oplus Z_p.
	\]
	where $r=\frac{p}{q}$.
	Next, we will compute the Poincar\'e dual of $c_1(f^1_*(\xi_1))$ and express it in terms of $[\mu]$ and $[\nu]$. Consider the contact $(\pm1)$--surgery diagram for $(M_1,\xi_1)$. Let $-\frac{r-1}{r-2}= [r_1,\ldots,r_k]$, $L_0$ the Legendrian trefoil in Figure~\ref{fig:legendrian1}, and $L_1,...,L_k$ its push-offs after the conversion.
	We also denote the Legendrian unknots in Figure~\ref{fig:legendrian1} by $L'_1, L'_2, \cdots, L'_{n-1}$ from the bottom to the top. According to the results from \cite{DGS:surgery, Gompf:stein}, we have 
	\begin{align}\label{eq:c1xi1}
		\PD(c_1(\xi_1))=\sum_{i=0}^k \rot(L_i)[\mu_i] + \sum_{i=1}^{n-1}\rot(L'_i)[\mu'_i], 
	\end{align}
	where $\rot(L'_i) = 0$ for $1 \leq i \leq n-1$, $\rot(L_0)=0$ and $\rot(L_i)$ for $1 \leq i \leq k$ depends on the choice of stabilizations. We now keep track of this homology class under the diffeomorphisms induced by surgery operations. 

	First, consider the chain of unknots $L'_1, \ldots, L'_{n-1}$. We apply slam-dunk moves to remove all $L'_i$ for $2 \leq i \leq n-1$. Suppose $L'_i$ is the last component of the chain. We investigate the effect of a slam-dunk move to $L'_i$. We claim that the homology classes will change under the induced diffeomorphism as follows: 
	\begin{align*}
		[\mu'_i] &\mapsto -2[\mu'_{i-1}] - [\mu'_{i-2}]
	\end{align*}	
	and all other $[\mu_j]$ and $[\mu'_{j}]$ remained unchanged. To see this, slide $\mu'_{i}$ over $L'_{i-1}$ and obtain a $(1,-2)$-cable of $L'_{i-1}$ as shown in the top of Figure~\ref{fig:surgeryOperations}. Thus we have $[\mu'_{i}] = -2[\mu'_{i-1}] + [\lambda'_{i-1}]$ where $\lambda'_{i-1}$ is a Seifert longitude of $L'_{i-1}$. Then apply the slam-dunk move to $L'_{i}$ and remove it.  Since the slam-dunk move is a local operation, it only affects a small neighborhood of $D'_{i} \cup L'_{i-1}$ where $D'_{i}$ is a Seifert disk of $L'_{i}$. Thus it does not change the isotoped $\mu'_{i}$. Also, $\lambda'_{i-1}$ is isotopic to $-\mu'_{i-2}$ since $\mathrm{lk}(L_i, L_{i-1}) = \tb(L_{i-1}) = -1$. Thus the result follows. 

	Now we turn our attention to $L_1,\ldots,L_k$. First, notice that the smooth surgery coefficient of $L_0$ is $2 = \tb(L_0) + 1$, and all other $L_i$ is $\tb(L_i) - 1$. We slide $L_k$ over $L_{k-1}$ and $L_k$ becomes a meridian of $L_{k-1}$. Since $\mathrm{lk}(L_{k},L_{k-1})=\tb(L_{k-1})$, the surgery coefficient becomes 
	\begin{align*}
		(\tb(L_k)-1) + (\tb(L_{k-1})-1) - 2\tb(L_{k-1}) &= \tb(L_k) - \tb(L_{k-1}) - 2 \\
		&= r_k.
	\end{align*}
	The second equality follows from the fact that both $|r_k + 2|$ and $|\tb(L_k) - \tb(L_{k-1})|$ are the difference between the number of stabilizations on $L_k$ and $L_{k-1}$. We keep sliding $L_i$ over $L_{i-1}$ for all $1 \leq i \leq k$ consecutively. Then we obtain a chain of unknots where $\mathrm{lk}(L_i,L_{i-1}) = -1$ for $1 \leq i \leq k$. Notice that the surgery coefficient of $L_i$ is
	\begin{align*}
	\begin{cases}
		\tb(L_i) - \tb(L_{i-1}) - 2 = r_{i} & 2 \leq i \leq k,\\ 
		\tb(L_i)-\tb(L_{i-1}) = r_i+1 & i=1,\\
		\tb(L_i) + 1 = 2  & i =0.
	\end{cases}
	\end{align*}	
	We can now perform a sequence of slam-dunk moves to eliminate all $L_i$ for $1 \leq i \leq k$. Suppose $L_i$ is the last component of the chain. As what we did to $L'_i$ in the previous paragraph, we can keep track of the homology classes under the slam-dunk moves on $L_i$ and the homology classes change as follows:
	\begin{align*}
	\begin{cases}
  	[\mu_{i}] \mapsto r_{i-1}[\mu_{i-1}] - [\mu_{i-2}] &  3 \leq i \leq k\\
  	[\mu_{i}] \mapsto (r_{i-1} + 1)[\mu_{i-1}] - [\mu_{i-2}] &  i=2\\
  	[\mu_{i}] \mapsto 2[\mu_{i-1}] &  i=1
	\end{cases}
	\end{align*}
	and all other $[\mu_j]$ and $[\mu'_j]$ remain unchanged.
	
	Finally, we apply a right-handed Rolfsen twist along $L'_1$ and obtain the surgery diagram in Figure~\ref{fig:Weeks1}. As shown in the bottom left of Figure~\ref{fig:surgeryOperations}, the meridian $\mu'_1$ becomes $(1,1)$-cable of $L'_1$ after the Rolfsen twist. Let $\lambda'_1$ be a Seifert longitude of $L'_1$. Since the linking number between the two link components is $0$, the homology classes will change under the induced diffeomorphism as follows: 
	\begin{align*}
		[\mu'_1] &\mapsto [\mu'_1] + [\lambda'_1] = [\mu'_1] + [\mu_0] - [\mu_0]  = [\mu'_1],
	\end{align*}
	and $[\mu_0]$ remains unchanged.

	We can now naturally identify $\mu_0$ and $\mu'_1$ with $\nu$ and $\mu$, respectively. Altogether, we have 
 	\begin{align*}
		\PD(c_1(f^1_*(\xi_1))) &= f^1_*(\PD(c_1(\xi_1))\\
		&= f^1_*\left(\sum_{i=0}^k \rot(L_i)[\mu_i] + \sum_{i=1}^{n-1}\rot(L'_i)[\mu'_i]\right)\\
		&= f^1_*\left(\sum_{i=0}^k \rot(L_i)[\mu_i]\right)\\
		&= \sum_{i=0}^k \rot(L_i)\cdot f^1_*\left([\mu_i]\right)\\
		&= 0[\mu] + m[\nu]
 	\end{align*}
	for some $m\in \mathbb{Z}$. The first and fourth equalities follow from the naturality, the second equality follows from (\ref{eq:c1xi1}), the third equality follows from the fact that $\rot(L'_i)=0$ for $1 \leq i \leq n-1$ and the fifth equality follows from the discussions about the effect of surgery operations in the previous paragraphs.
 
	Similarly, we can calculate $c_1(f^2_*(\xi_2))$. Consider the contact $(\pm1)$--surgery diagram for $(M_2,\xi_2)$. Let $L_1$ be the Legendrian trefoil in Figure~\ref{fig:legendrian2}. Since we have 
	\begin{align*}
		\frac{n-1}{1-(n-1)} = -\frac{n-1}{n-2}= [\overbrace{-2,\cdots, -2}^{n-2}], 
	\end{align*}
	there are $(n-2)$ push-offs of $L_1$ after the conversion. Denote them by $L_2,L_3,\cdots,L_{n-1}$. Recall that during the conversion, we stabilize $L_2$ once and push it off $n-3$ times. Thus the rotation numbers of $L_i$ are
	\begin{align}\label{eq:rot}
		\rot(L_1) = 0, \rot(L_2) = \rot(L_3)= \cdots =\rot(L_{n-1}) = \pm1.
	\end{align}
 	Let $L'_0,\ldots, L'_k$ be the Legendrian unknot in Figure~\ref{fig:legendrian2} and its push-offs after the conversion.  As before, we have  
	\begin{align}\label{eq:c1xi2}
		\PD(c_1(\xi_2))=\sum_{i=1}^{n-1} \rot(L_i)[\mu_i] + \sum_{i=0}^{k}\rot(L'_i)[\mu'_i].
	\end{align}

	We first consider $L_1,\ldots,L_{n-1}$. We start by sliding all $L_i$ over $L_1$ for $2 \leq i \leq n-1$. This produces an unlink of $(n-2)$ components which are meridians of $L_1$ and the surgery coefficient of each component becomes $-1$. The handle-slide operation on $L_i$ induces a diffeomorphism and the homology classes will change under this diffeomorphism as follows:
	\[
		[\mu_1] \mapsto [\mu_1] + [\mu_i]
	\]	
	and all other $[\mu_j]$ and $[\mu'_j]$ remain unchanged. This is clear from the bottom right of Figure~\ref{fig:surgeryOperations}. Now apply right-handed Rolfsen twist along each $L_i$ for $2 \leq i \leq n-1$ and get rid of it. As shown in the bottom left of Figure~\ref{fig:surgeryOperations}, each $\mu_i$ becomes a $(1,1)$-cable of $L_i$. Let $\lambda_i$ be a Seifert longitude of $L_i$. Since $\lambda_i$ is a meridian of $L_1$ and $\mu_i$ becomes contractible, the homology classes will change under the induced diffeomorphism as follows:  
	\[
		[\mu_i] \mapsto [\mu_i] + [\lambda_i] = [\mu_1]
	\]
	and all other $[\mu_j]$ and $[\mu'_j]$ remain unchanged.

	We now perform a sequence of handle-slides to make $L'_0,\ldots,L'_k$ a linear chain of unknots and then perform slam-dunk moves to get rid of all $L'_i$ for $1 \leq i \leq k$. Thus we apply a similar argument as before to keep track of the homology classes. The only important fact here is that these operations only affect $\mu'_i$ and do not affect any $\mu_i$. 	

	Finally, we apply a right-handed Rolfsen twist along $L'_0$ and obtain the surgery diagram in the left of Figure~\ref{fig:Weeks2}. Then swap the link components and obtain the diagram in Figure~\ref{fig:Weeks1}. As in the case of $\xi_1$, this Rolfsen twist does not affect any homology class.  

	We can now naturally identify $\mu_1$ and $\mu'_0$ with $\mu$ and $\nu$, respectively. Altogether, we have 
	\begin{align*}
	 \PD(c_1(f^2_*(\xi_2))) &= f^2_*(\PD(c_1(\xi_2))\\
	 &= f^2_*\left(\sum_{i=1}^{n-1} \rot(L_i)[\mu_i] + \sum_{i=0}^k\rot(L'_i)[\mu'_i]\right)\\
	 &= \sum_{i=2}^{n-1} \pm f^2_*\left([\mu_i]\right) + \sum_{i=0}^k \rot(L'_i)\cdot f^2_*\left([\mu'_i]\right)\\
	 &= \pm(n-2)[\mu] + h[\nu]
	\end{align*}
 for some $h \in \mathbb{Z}$. The first equality follows from the naturality, the second equality follows from (\ref{eq:c1xi2}), the third equality follows from (\ref{eq:rot}) and the fourth equality follows from the discussions about the effect of surgery operations in the previous paragraphs.

 Since we assume $n > 2$, we can conclude that $f^1_*(\xi_1)$ and $f^2_*(\xi_2)$ are distinguished by their first Chern classes, which completes the proof of the claim.
\end{proof}

\section{The upper bounds}\label{sec:upper-bounds}
In this section, we determine the upper bounds of the number tight contact structures on $M(n,r)$ for $n \geq 5$ and $r \in \mathcal{R}_+ \cup \mathbb{Q}_-$. Combining it with the lower bounds in Section~\ref{sec:lower-bounds}, we complete the proof of Theorem~\ref{thm:main}. 

\subsection{Convex decompositions of \texorpdfstring{$M(n,r)$}{M(n,r)}} \label{subsec:decomp} 
The first step is to decompose $M(n,r)$ into two pieces with convex torus boundary. Consider the surgery diagram for $M(n,r)$ shown in Figure~\ref{fig:Weeks1}. If we perform the surgery only on $K_1$, then the result will be $L(n,n-1)$. Moreover, it is well known ({\it c.f.}~\cite{HMW:fiber, Morimoto:fiber}) that $K_2$ is a genus one fibered knot in $L(n,n-1)$. Let $N$ be a neighborhood of $K_2$ in $L(n,n-1)$, $C_n$ the complement of $N$, $\Sigma$ a fiber surface of $C_n$ and $\phi_n$ the monodromy of $C_n$. Then there is a symplectic basis of $H_1(\Sigma)$ such that the action of the monodromy $\phi_n$ on $H_1(\Sigma)$ is given by 
\begin{align*}
	\phi_n = \matrixp{-n+2}{1}{-1}{0} =  \matrixp{1}{0}{-1}{1} \matrixp{1}{1}{0}{1} \matrixp{1}{0}{-1}{1}^{n-1}.
\end{align*}
We can also assume that $\phi_n$ is the identity near the boundary. Notice that $\phi_n$ is pseudo-Anosov for $n \geq 5$, and the fractional Dehn twist coefficient $c(\phi_n) \in (0,1)$ so $\phi_n$ is right-veering. Since $M(n,r)$ is a Dehn filling of $C_n$, we can write  
\[
	M(n,r) = C_n \cup N_r
\]
where $N_r$ is a solid torus with the meridional slope $r$, with respect to the coordinate system induced by $K_1$. We can also write $C_n$ as 
\[
	C_n = \Sigma \times [0,1] / (x,1) \sim (\phi_n(x),0).
\]

Let $\xi$ be a contact structure on $M(n,r)$ and $L$ a Legendrian knot in $N_r$ isotopic to the core of $N_r$. We can assume that $N_r$ is a standard neighborhood of $L$ after a perturbation, so $\partial N$ has convex boundary with two dividing curves of some slope $s$ (again, we measure slopes on $\partial N_r = -\partial C_n$ using the coordinates induced by $K_1$). We will denote $C_n(s)$ and $N_r(s)$ when $C_n$ and $N_r$ have contact structures with convex boundaries with two dividing curves of slope $s$.

As was done in \cite{CM:figure-eight}, we will \emph{thicken} $C_n(s)$ into a standard form, which means finding $C_n(s') \subset C_n(s)$ such that $C_n(s) \setminus C_n(s') = T^2 \times I$. We say that $C_n(s)$ \emph{thickens} to $C_n(s')$, or $C_n(s)$ \emph{thickens to slope $s'$}. Also, we say $C_n(s)$ \emph{does not thicken} if $C_n(s)$ thickens to $s'$ implies $s = s'$. We will thicken $C(s)$ by attaching bypasses to $-\bd C_n(s)$. We will find such bypasses by finding boundary parallel dividing arcs on a fiber surface $\Sigma \subset C_n(s)$. Due to our slope convention and the choice of orientation, the slope increases after attaching a bypass to $-\bd C_n(s)$. 

Let $\mathcal{S}(r)$ be the set of slopes $s$ that is clockwise of $r$, anticlockwise of $\infty$ on the Farey graph and connected with $r$ by an edge. For $r\in \mathcal{R}_+$, we need the following two lemmas.

\begin{lemma}\label{lem:thickeningPositive}
	Suppose $n \geq 5$, $r \in \mathcal{R}_+$ and $s \in \mathcal{S}(r)$. Then $C_n(s)$ thickens to $C_n(\infty)$.
\end{lemma}

Once we thicken $C_n(s)$ to $C_n(\infty)$, we will determine an upper bound for the number of tight contact structures on $C_n(\infty)$. 

\begin{lemma}\label{lem:tightInfty}
	Suppose there is no boundary parallel half Giroux torsion. Then there exist at most four tight contact structures on $C_n(\infty)$ up to isotopy (not fixing the boundary pointwise but preserving it setwise). Two of them thicken further to $C_n(1)$. 
\end{lemma}

\begin{remark}\label{rmk:boudnary-twisting}
	Since we will classify tight contact structures on closed manifolds, which are obtained by gluing a solid torus to the boundary of $C_n(\infty)$, it is enough to classify tight contact structures on $C_n(\infty)$ up to isotopy fixing the boundary setwise.
\end{remark}

For $r \in \mathbb{Q}_-$, we need the following two lemmas.

\begin{lemma}\label{lem:thickeningNegative}
	Suppose $n \geq 5$, $r \in \mathbb{Q}_-$, $s \in \mathcal{S}(r)$ and $s\neq 0$. Then $C_n(s)$ thickens to $C_n(1)$.
\end{lemma}

According to Lemma~\ref{lem:thickeningNegative}, we can write $C_n(s) = T \cup C_n(1)$ for $s< 0$, where $T = T^2 \times [0,1]$ with convex boundary and dividing slopes $s_0 = s$ and $s_1 = 1$.

\begin{lemma}\label{lem:tightReciprocal}
	Suppose there is no boundary parallel half Giroux torsion. Then there exist at most four tight contact structures on $C_n(-\frac1k)$ for $k \in \mathbb{N}$ up to isotopy (not fixing the boundary pointwise but preserving it setwise). Moreover, the (possibly) tight contact structures on $C_n(-\frac1k)$ are determined by their restriction to $T$.
\end{lemma}

\begin{remark}
	In fact, there are exactly three tight contact structures on $C_n(-1)$ without boundary parallel Giroux torsion but we will not use this fact in this paper.
\end{remark}

We also need to count the number of tight contact structures on $N_r(\infty)$ and $N_r(1)$ using Proposition~\ref{prop:solid-torus}.

\begin{lemma}\label{lem:solid-torus}
	For any $r \in \mathbb{Q}$, there are $\Phi(r)$ tight contact structures on $N_r(\infty)$ up to isotopy fixing the boundary pointwise (hence fixing the characteristic foliation).
	
	Also, there are $\Psi(r)$ tight contact structures on $N_r(1)$ up to isotopy fixing the boundary pointwise.
\end{lemma}

Combining these lemmas, we can obtain the upper bounds.

\begin{theorem}\label{thm:upperBounds} 
	Suppose $n \geq 5$ and $r\in \mathcal{R}_+ \cup \mathbb{Q}_-$. 
	\begin{itemize}
		\item $M(n,r)$ supports at most $2\Phi(r)+\Psi(r)$ tight contact structures up to isotopy if $r \in \mathcal{R}_+$.
		\item $M(n,r)$ supports at most $\Psi(r)$ tight contact structures up to isotopy if $r \in \mathbb{Q}_-$.
	\end{itemize}
\end{theorem}

\begin{proof}
	We first consider the case $r \in \mathcal{R}_+$. By Lemma~\ref{lem:thickeningPositive} and Lemma~\ref{lem:tightInfty}, $C_n(s)$ thickens to slope $\infty$ or $1$ for any $s \in \mathcal{S}(r)$. Thus any tight contact structure on $M(n,r)$ can be decomposed as a union of a tight contact structure on $C_n(\infty)$ and a tight contact structure on $N_r(\infty)$, or a union of a tight contact structure on $C_n(1)$ and a tight contact structure on $N_r(1)$. In the case that $C_n(s)$ thickens to $C_n(\infty)$ and not further, there are at most two tight contact structures on $C_n(\infty)$ by Lemma~\ref{lem:tightInfty}. Also, there are $\Phi(r)$ tight contact structure on $N_r(\infty)$ by Lemma~\ref{lem:solid-torus}. These lead to $2\Phi(r)$ (possibly) tight contact structures on $M(n,r)$. In the case that $C_n(s)$ thickens to $C_n(1)$, we can decompose $N_r(1)$ into a union of $N_r(-1)$ and $T$, where $T = T^2 \times I$ with dividing slopes $-1$ and $1$. By Lemma~\ref{lem:tightReciprocal}, a tight contact structure on $T \cup C_n(1) = C_n(-1)$ is completely determined by $T$. Also, there are $\Psi(r)$ tight contact structure on $N_r(-1) \cup T = N_r(1)$ by Lemma~\ref{lem:solid-torus}. These lead to the $\Psi(r)$ (possibly) tight contact structures on $M(n,r)$.

	Next, we consider the case $r \in \mathbb{Q}_-$. By Lemma~\ref{lem:thickeningNegative}, $C_n(s)$ thickens to slope $1$ for any $s \in \mathcal{S}(r)$ and $s\neq 0$. If $s = 0$, there is $N_r(s') \subset N_r(0)$ for $r < s' < 0$, so we can decompose $M(n,r)$ into $N_r(s')$ and $C_n(s')$. By Lemma~\ref{lem:thickeningNegative}, we can thicken $C_n(s')$ to slope $1$. Thus any tight contact structure on $M(n,r)$ can be decomposed as a union of a tight contact structure on $C_n(1)$ and a tight contact structure on $N_r(1)$. We can decompose $N_r(1)$ into a union of $N_r(-\frac1m)$ and $T$, where $m \in \mathbb{N}$ and $T = T^2 \times I$ with dividing slopes $-\frac1m$ and $1$. By Lemma~\ref{lem:tightReciprocal}, a tight contact structure on $T \cup C_n(1) = C_n(-\frac1m)$ is completely determined by $T$. Also, there are $\Psi(r)$ tight contact structure on $N_r(-\frac1m) \cup T = N_r(1)$ by Lemma~\ref{lem:solid-torus}. These lead to the $\Psi(r)$ (possibly) tight contact structures on $M(n,r)$.
\end{proof}

We will prove the above lemmas in Section~\ref{subsec:infty}.

\subsection{Normalizing the dividing set on \texorpdfstring {$\Sigma$}{Sigma}}
We begin by normalizing the dividing set $\Gamma_{\Sigma}$ of $\Sigma$. Since we can freely modify the slopes of the Legendrian ruling curves on $N$, we assume $\Sigma$ is convex with Legendrian boundary. 

Etnyre and Honda \cite{EH:knots} normalized $\Gamma_{\Sigma}$ for any pseudo-Anosov monodromy. 

\begin{proposition}[Etnyre and Honda \cite{EH:knots}]\label{prop:normalization1}
	Suppose $\Gamma_{\Sigma}$ consists of $k > 0$ properly embedded arcs and $m$ closed curves. Then there is an isotopic copy of $\Sigma$ such that one of the following holds.
	\begin{enumerate}
		\item $k$ is odd.
		\begin{itemize}
			\item $m=1$ and all arcs are in the same isotopy class that is not boundary parallel. 
			\item $m=0$ and there are three distinct isotopy classes of arcs that are not boundary parallel. 
			\item there is a boundary parallel arc (possibly with other dividing curves).
		\end{itemize}
		\item $k$ is even.
		\begin{itemize}
			\item there is a boundary parallel arc (possibly with other dividing curves).
		\end{itemize}
	\end{enumerate}
\end{proposition}

When $k \geq 3$, Etnyre and Honda \cite{EH:knots} further normalized the dividing set on $\Sigma$ in the figure-eight knot complement. Their argument holds however for any pseudo-Anosov monodromy. 

Let $\phi$ be a pseudo-Anosov monodromy of $C = \Sigma \times [0,1] / (x,1) \sim (\phi(x),0)$. Consider the matrix representation of $\phi$ with respect to the symplectic basis of $\Sigma$. Then there are two eigenvectors $v_a$ and $v_r$, where $v_a$ is an attracting fixed point of $\phi$ and $v_r$ is a repelling fixed point of $\phi$. Denote by $s_a$ and $s_r$ the slopes of $v_a$ and $v_r$, respectively. There are two paths $P_c$ and $P_a$ in the Farey graph where $P_c$ is the clockwise path from $s_r$ to $s_a$ and $P_a$ is the anticlockwise path from $s_r$ to $s_a$.

\begin{lemma}[Etnyre and Honda \cite{EH:knots}]\label{lem:normalization2}
	Suppose $\Gamma_{\Sigma}$ consists of $k \geq 3$ properly embedded arcs and $m$ closed curves. Then there is an isotopic copy of $\Sigma$ such that one of the following holds.
	\begin{itemize}
		\item $m=0$ and there are three distinct isotopy classes of arcs with slopes $\{a,b,c\}$ in $P_c$. Moreover, $\{a,b,c\}$ forms the largest geodesic triangle in the Farey graph among geodesic triangles with vertices in $P_c$. 
		\item $m=0$ and there are three distinct isotopy classes of arcs with slopes $\{a,b,c\}$ in $P_a$. Moreover, $\{a,b,c\}$ forms the largest geodesic triangle in the Farey graph among geodesic triangles with vertices in $P_a$. 
		\item there is a boundary parallel arc (possibly with other dividing curves).
	\end{itemize}
\end{lemma}

Etnyre and Honda further convert these dividing sets into a simpler form.

\begin{lemma}[Etnyre and Honda \cite{EH:knots}]\label{lem:normalization3}
	Suppose that $\Gamma_{\Sigma}$ consists of $n$ arcs and there are three isotopy classes of arcs with slopes $\{a,b,c\}$ in $P_a$ (resp. $P_c$). Without loss of generality, assume that $a$ is the closest to $s_r$ in the Farey graph. Then there is another isotopic copy of $\Sigma$ such that 
	\begin{itemize}
		\item $\Gamma_{\Sigma}$ consists of one closed curve and $n$ arcs with slope $a$.
		\item there is a boundary parallel arc (possibly with other dividing curves).
	\end{itemize} 
\end{lemma}

Now we are ready to normalize the dividing set on $\Sigma$ in our $C_n(s)$.

\begin{proposition}\label{prop:normalizationCn}
	Suppose $\Sigma \subset C_n(s)$ and $\Gamma_{\Sigma}$ consists of $k \geq 2$ properly embedded arcs and $m$ closed curves. Then there is an isotopic copy of $\Sigma$ such that one of the following holds.
	\begin{enumerate}
		\item $k$ is odd.
		\begin{itemize}
			\item $m=1$ and all dividing arcs have slope $\infty$. 
			\item $m=0$ and there are three isotopy classes of dividing arcs with slopes $\{0, 1, \infty\}$, $\mathrm{mult}(1)=k-2$ and $\mathrm{mult}(0)=\mathrm{mult}(\infty)=1$. Here, $\mathrm{mult}(s)$ is the number of arcs with slope $s$. 
			\item there is a boundary parallel arc (possibly with other dividing curves).
		\end{itemize}
		\item $k$ is even, and there is a boundary parallel arc (possibly with other dividing curves).
	\end{enumerate}
\end{proposition}

\begin{proof}
	Assume that any isotopic copy of $\Sigma$ in $C_n(s)$ contains a boundary parallel dividing arc. By Proposition~\ref{prop:normalization1}, this implies that $k$ is odd. We first calculate the slopes of the fixed points $v^n_a$ and $v^n_r$ of $\phi_n$ for $n \geq 5$ and obtain 
	\[
		s_a^n = \frac{2}{n-2+\sqrt{n^2-4n}} \;\text{ and }\; s_r^n = \frac{2}{n-2-\sqrt{n^2-4n}}.
	\]
	Thus $0 < s_a^n < 1/2$ and $2 < s_r^n < \infty$. For any $n \geq 5$, the largest geodesic triangle with the vertices in $P_c^n$ is $\{\infty, -1, 0\}$, and the largest geodesic triangle with the vertices in $P_a^n$ is $\{\frac12, \frac23, 1\}$. By Lemma~\ref{lem:normalization2} and \ref{lem:normalization3}, there is an isotopic copy of $\Sigma$ such that 
	\begin{itemize}
		\item $\Gamma_{\Sigma}$ consists of one closed curve and $k$ arcs with slope $\infty$, or
		\item $\Gamma_{\Sigma}$ consists of one closed curve and $k$ arcs with slope $1$.
	\end{itemize}
	We claim that for the second case there is another isotopic copy of $\Sigma$ such that there is a boundary parallel arc, or there are three isotopy classes of dividing arcs with slopes $\{0, 1, \infty\}$, $\mathrm{mult}(1)=k-2$ and $\mathrm{mult}(0)=\mathrm{mult}(\infty)=1$. This claim will compete the proof.
	
	We remain to prove the claim. Let $U$ be a small $I$-invariant neighborhood of $\Sigma$ in $C_n(s)$. Then $C_n(s) \setminus U = \Sigma\times[0,1]$. Let $\Sigma_i:=\Sigma\times\{i\}$ for $i = 0,1$ and $\Gamma_i$ be a dividing set of $\Sigma_i$. Then we have $\Gamma_1 = \Gamma_{\Sigma}$ and $\Gamma_0 = \phi_n(\Gamma_1)$. 
	
	We first consider the case $n > 5$. Take a closed curve $c$ on $\Sigma_1$ with slope $\frac{1}{-n+3}$. Consider an annulus $c \times [0,1]$ and let $c_i = c \times \{i\}$ for $i=0,1$. Assume $c_i$ intersects $\Gamma_i$ minimally. Then $|c_1 \cap \Gamma_1| = (n-2)(k+1)$ and $|c_0 \cap \Gamma_0| =  0$. Thus there is a bypass on $A$ along $c_1$. If the bypass does not straddle the closed dividing curve, attaching the bypass results in a boundary parallel arc (here, we attach the bypass from the back). If the bypass straddles the closed dividing curve, attaching the bypass results in the dividing set with slope $\{0,1,\infty\}$ where $\mathrm{mult}(1)= k-2$ and $\mathrm{mult}(0)= \mathrm{mult}(\infty)=1$. 
	
	Now if $n = 5$, take a closed curve $c$ on $\Sigma_1$ with slope $\frac{2}{5}$. Consider an annulus $c \times [0,1]$ and let $c_i = c \times \{i\}$ for $i=0,1$. Assume $c_i$ intersects $\Gamma_i$ minimally. Then $|c_1 \cap \Gamma_1| = 3(k+1)$ and $|c_0 \cap \Gamma_0| =  k+1$. Thus there is a bypass on $A$ along $c_1$. If the bypass does not straddle the closed dividing curve, attaching the bypass results in a boundary parallel arc. If the bypass straddles the closed dividing curve, attaching the bypass results in the dividing set with slope $\{0,1,\infty\}$ where $\mathrm{mult}(1)= k-2$ and $\mathrm{mult}(0)= \mathrm{mult}(\infty)=1$. 
\end{proof}

We also consider the case for $k=1$ without boundary parallel half Giroux torsion.

\begin{proposition}\label{prop:normalizationCn1}
	Suppose there is no boundary parallel half Giroux torsion in $C_n(s)$ and $\Gamma_{\Sigma}$ consists of one properly embedded arc and $m$ closed curves. Then there exists an isotopic copy of $\Sigma$ in $C_n(s)$ such that 
	\begin{itemize}
		\item one arc and one closed curve with slope $\infty$.
		\item one boundary parallel arc without any other dividing curves. 
	\end{itemize}
\end{proposition}

We will prove this proposition in Section~\ref{subsec:infty}.

\subsection{Thickening \texorpdfstring{$C_n(s)$}{Cn(s)}}\label{subsec:thickening}
Here, we will prove Lemma~\ref{lem:thickeningPositive}. We will thicken $C_n(s)$ to $C_n(s')$ by finding a bypass for $-\partial C_n(s)$ along $\partial \Sigma$. Recall that due to our slope conventions, this bypass has slope $0$ on $\partial C_n(s)$ and $s$ will change clockwise on the Farey graph.

\begin{figure}[htbp]
	\begin{center}
	\begin{overpic}[scale=0.9,tics=20]{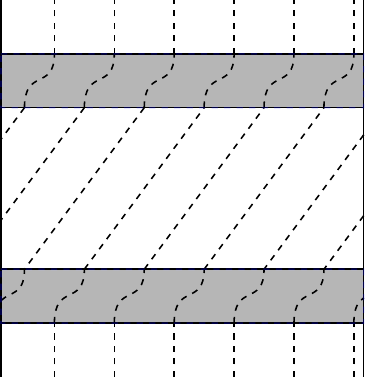}
	\put(-50,150){$\Sigma\times\{1\}$}
	\put(-50,8){$\Sigma\times\{0\}$}
	\put(10,150){$1$}
	\put(35,150){$2$}
	\put(60,150){$3$}
	\put(85,150){$4$}
	\put(112,150){$5$}
	\put(138,150){$6$}
	\put(10,8){$1$}
	\put(35,8){$2$}
	\put(60,8){$3$}
	\put(85,8){$4$}
	\put(112,8){$5$}
	\put(138,8){$6$}
	\end{overpic}
	\caption{Edge rounding when $s=3$. The left side is identified with the right side.}
	\label{fig:positive-edge-rounding}
	\end{center}
\end{figure}
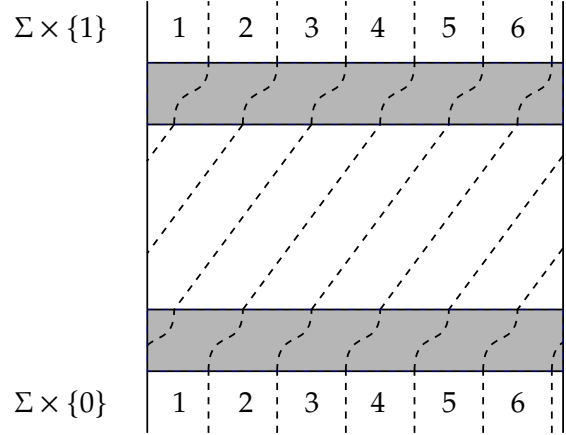

\begin{proposition}\label{prop:positiveThickening}
	Suppose $s \geq 2$ and $s \notin \{\frac{4k+1}{k}\mid k \in \mathbb{N}\}$. Then there is an isotopic copy of $\Sigma$ in $C_n(s)$ such that $\Gamma_{\Sigma}$ contains a boundary parallel arc.
\end{proposition}

\begin{proof}
  Let $s=\frac{m}{k}$ for a pair of relatively prime positive integers $m, k$. There are $m$ properly embedded dividing arcs on $\Sigma$. We cut $C_n(s)$ along $\Sigma$ to obtain a genus-$2$ handlebody $\Sigma \times [0,1]$ and round the edges to obtain a smooth convex boundary. Let $\Sigma_i$ be $\Sigma \times \{i\}$ for $i=0,1$ and $\Gamma_i$ the dividing set on $\Sigma_i$. Then $\Gamma_0 = \phi_n(\Gamma_1)$. After rounding the edges, the entire dividing set only contains closed curves, but we will keep calling the dividing curves that pass through $\bd\Sigma$ dividing arcs for convenience.  
    
  For $i=0,1$, the dividing arcs in $\Gamma_i$ divides $\bd\Sigma_i$ into $2m$ intervals, which we will label by $1$ through $2m$. The dividing curves on $\partial \Sigma \times [0,1]$ connect the $i$-th interval on $\Sigma_1$ to the $(i-2k-1)$-th interval (mod $2m$) on $\Sigma_0$. See Figure~\ref{fig:positive-edge-rounding} for example.

	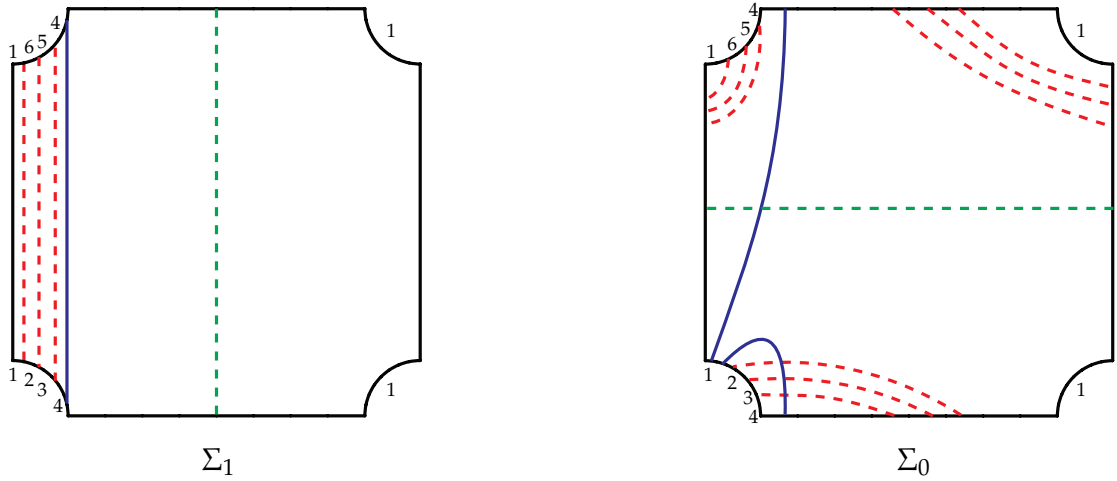
\begin{figure}[htbp]
	\begin{center}
		\begin{overpic}[tics=20]{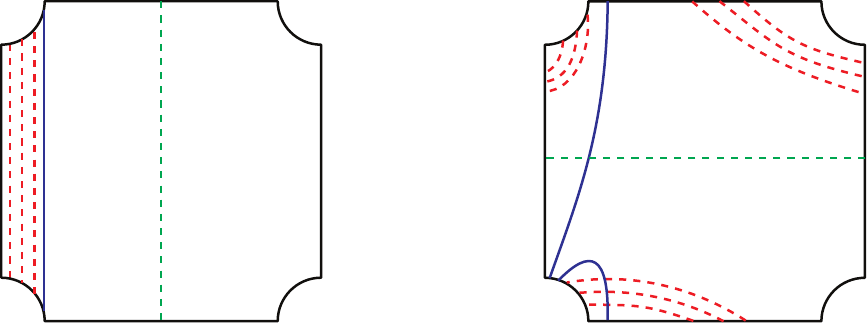}
			\put(-1,14){\tiny$1$}
			\put(5,12){\tiny$2$}
			\put(10,8){\tiny$3$}
			\put(16,2){\tiny$4$}

			\put(15,147){\tiny$4$}
			\put(10,140){\tiny$5$}
			\put(5,138){\tiny$6$}
			\put(-1,136){\tiny$1$}

			\put(142,8){\tiny$1$}
			\put(142,144){\tiny$1$}

			\put(262,14){\tiny$1$}
			\put(271,11){\tiny$2$}
			\put(277,5){\tiny$3$}
			\put(279,-2){\tiny$4$}

			\put(279,152){\tiny$4$}
			\put(276,145){\tiny$5$}
			\put(271,139){\tiny$6$}
			\put(263,136){\tiny$1$}

			\put(403,8){\tiny$1$}
			\put(403,144){\tiny$1$}

			\put(72,-20){\large$\Sigma_1$}
			\put(335,-20){\large$\Sigma_0$}
		\end{overpic}
		\vspace{0.8cm}
		\caption{$\Sigma_1$ and $\Sigma_0$ in $C_n(3)$. The dotted lines are dividing curves. The blue lines are $\alpha \times \{1\}$ and $\alpha \times\{0\}$. In each drawing, the top and bottom sides are identified, and so are the left and right sides. In this figure, as well as in the rest of the paper, we draw $\Sigma_0$ with the orientation induced from the fibration $C_n(s)$, which disagrees with that induced by $\Sigma \times [0,1]$. Also in this picture (and this picture only), we have enumerated the regions of $\bd\Sigma$ to better see the twisting of the dividing set as it runs over $\bd \Sigma \times [0,1]$.}
		\label{fig:Sigma1}
	\end{center}
	\end{figure}

	Suppose there is no boundary parallel dividing arc in $\Gamma_i$ for $i=0,1$. By Proposition~\ref{prop:normalizationCn} and Proposition~\ref{prop:normalizationCn1}, we need to consider two cases.

	(1) {\it $\Gamma_1$ contains $m$ arcs and one closed curve with slope $\infty$}: Let $\alpha$ be an arc on $\Sigma_1$ with slope $\infty$ that does not intersect $\Gamma_1$. Let $D_{\alpha}:= \alpha \times [0,1]$ be a compressing disk for $\Sigma \times [0,1]$. Perturb $D_{\alpha}$ so that its boundary does not intersect any dividing curve on $\partial \Sigma \times [0,1]$. This results in a shift of the basepoints of $\alpha$ on $\Sigma_0$ by $(2k+1)$ intervals following the negative orientation of $\partial \Sigma$, see Figure~\ref{fig:Sigma1} for example. Thus $\bd D_{\alpha}$ only intersects the dividing curves on $\Sigma_0$. We perturb $D_{\alpha}$ further so that it is convex with Legendrian boundary. We can also assume that $\bd D_{\alpha}$ intersects $\Gamma_0$ minimally, so it will intersect the closed dividing curve exactly once. Also, this intersection point separates $\alpha \times \{0\}$ into two sides. On one side it intersects $|2k+1|$ dividing arcs, and on the other side it intersects $|m-(2k+1)|$ dividing arcs. Since the total number of intersection points is greater than $2$, there is a bypass in $D_{\alpha}$ by Theorem~\ref{thm:DiskImbalance}. The difference between the number of intersections is:
  \begin{equation*}
    d= ||m-2k-1| - |2k+1|| = \begin{cases} |m-4k-2| & m \geq 2k+1\\ |m| & m < 2k + 1 \end{cases} 
  \end{equation*}
  If $d \neq 1$, then the dividing curves on $D_{\alpha}$ cannot be nested as shown in the right drawing of Figure~\ref{fig:disk1}. Thus there is always a bypass which does not straddle the closed dividing curve. After attaching this bypass, we obtain an isotopic copy of $\Sigma$ containing a boundary parallel dividing arc.

	\begin{figure}[htbp]
		\begin{center}
			\begin{overpic}[tics=20]{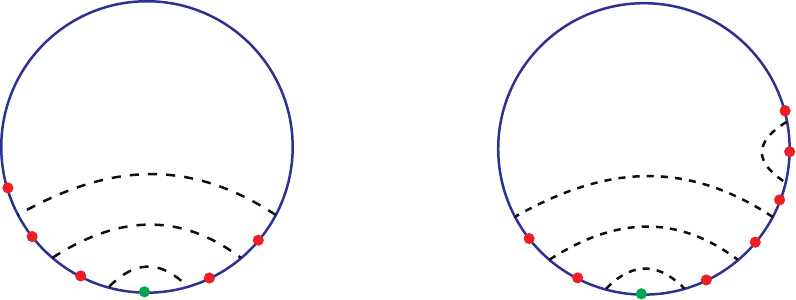}
			\end{overpic}
			\caption{Dividing curves in $D_{\alpha}$.}
			\label{fig:disk1}
		\end{center}
	\end{figure}
       
	If $d=1$, then $m=1$, $m=4k+1$ or $m=4k+3$. Since we assume $s \geq 2$ and $s \neq \frac{4k+1}{k}$, we only need to consider the case $m = 4k+3$. In this case, there are $2k+2$ dividing curves on $D_{\alpha}$ and they can be nested as shown in the left drawing of Figure~\ref{fig:disk1}. Attach all $2k+1$ bypasses to $\Sigma_0$ in sequence to obtain $\Sigma_{1/2}$ with dividing slopes $\{\infty,-1,0\}$ such that $\mathrm{mult}(\infty)=m-2$ and $\mathrm{mult}(-1)=\mathrm{mult}(0)=1$, see the right drawing of Figure~\ref{fig:Sigma3}. Then we can find an overtwisted disk as shown in Figure~\ref{fig:Sigma3}. Thus we can exclude this case too.

	\begin{figure}[htbp]
	\begin{center}
		\begin{overpic}[tics=20]{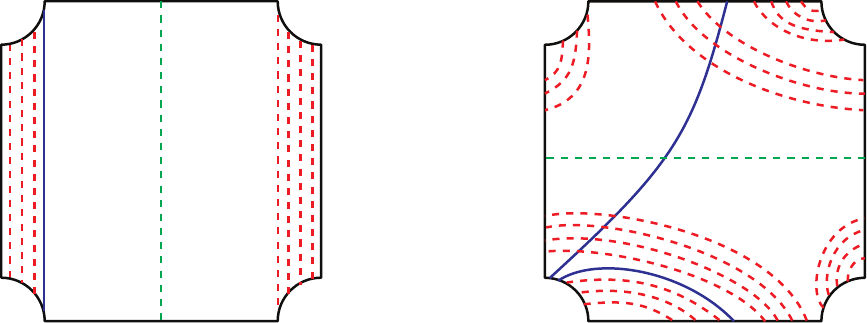}
			\put(72,-20){\large$\Sigma_1$}
			\put(335,-20){\large$\Sigma_0$}
		\end{overpic}
		\vspace{0.8cm}
		\caption{$\Sigma_1$ and $\Sigma_0$ in $C_n(7)$. The dotted lines are dividing curves and the blue lines are $\alpha\times\{1\}$ and $\alpha\times\{0\}$}
		\label{fig:Sigma2}
	\end{center}
	\end{figure}

	\begin{figure}[htbp]
	\begin{center}
		\begin{overpic}[tics=20]{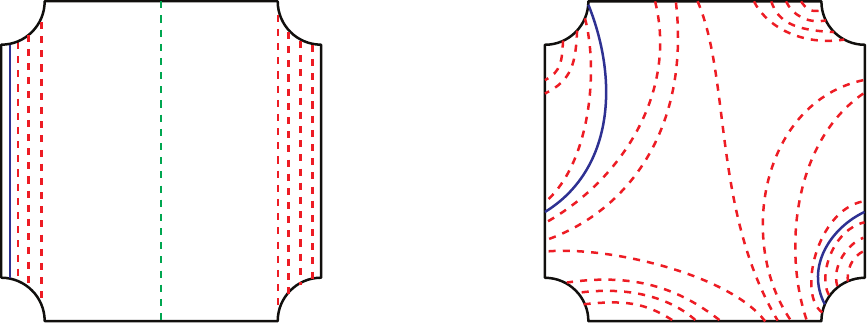}
			\put(72,-20){\large$\Sigma_1$}
			\put(335,-20){\large$\Sigma_{\frac12}$}
		\end{overpic}
		\vspace{0.8cm}
		\caption{$\Sigma_1$ and $\Sigma_{\frac12}$ in $C_n(7)$. The dotted lines are dividing curves and the blue lines are the intersections between an overtwisted disk and $\Sigma \times \{\frac12,1\}$}
		\label{fig:Sigma3}
	\end{center}
	\end{figure}

	(2) {\it $m \geq 3$ and $\Gamma_1$ contains arcs with slope $\{0,1,\infty\}$ where $\mathrm{mult}(0) = \mathrm{mult}(\infty) = 1$ and $\mathrm{mult}(1)=m-2$}: Let $\beta$ be an arc on $\Sigma_1$ with slope $\infty$ that does not intersect $\Gamma_1$, and $D_{\beta}:= \beta \times [0,1]$ be a compressing in $\Sigma \times [0,1]$. We perturb $D_{\beta}$ so that it is convex with Legendrian boundary, $\bd D_{\beta}$ does not intersect the dividing curves on $\bd \Sigma \times [0,1]$, and $\beta \times \{0\}$ intersects $\Gamma_0$ minimally, see Figure~\ref{fig:Sigma4}. Since $\beta\times\{0\}$ intersects more than four dividing curves in $\Gamma_0$, we can find a bypass in $D_{\beta}$ whose attaching arc lies on $\beta \times \{0\}$ by Theorem~\ref{thm:DiskImbalance}. From now on, we call dividing curves on $\Sigma_0$ with slopes $0$, $\frac{1}{n-2}$ and $\frac{1}{n-3}$ black, red and green dividing curves, respectively, see the right drawing of Figure~\ref{fig:Sigma4}. There are two cases we need to consider.  

	\begin{figure}[htbp]
	\begin{center}
		\begin{overpic}[tics=20]{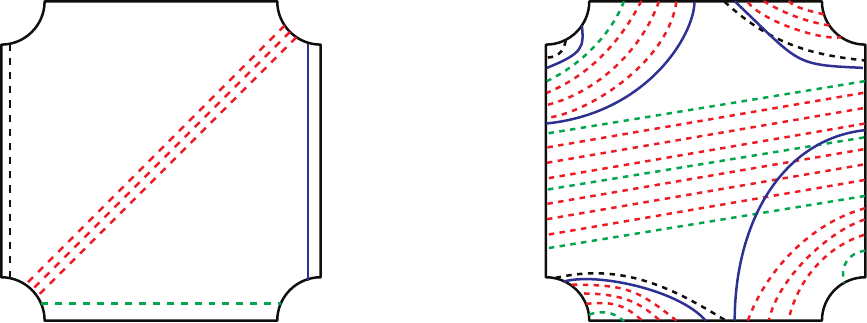}
			\put(72,-20){\large$\Sigma_1$}
			\put(335,-20){\large$\Sigma_0$}
		\end{overpic}
		\vspace{0.8cm}
		\caption{$\Sigma_1$ and $\Sigma_0$ in $C_n(\frac52)$. The dotted lines are dividing curves and the blue lines are $\alpha \times \{1\}$ and $\alpha \times\{0\}$}
		\label{fig:Sigma4}
	\end{center}
	\end{figure}

	\begin{figure}[htbp]
		\begin{center}
			\begin{overpic}[tics=20]{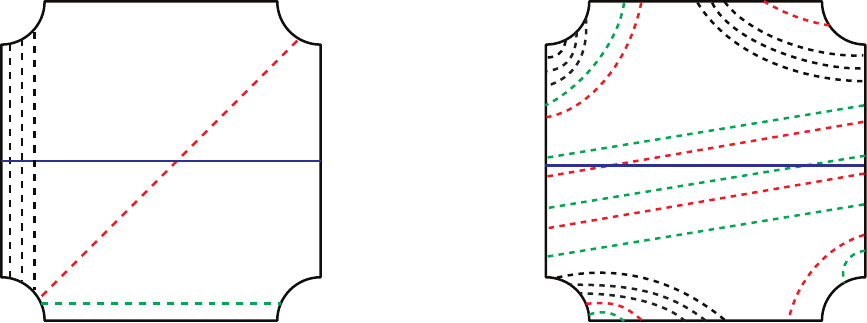}
				\put(72,-20){\large$\Sigma_1$}
				\put(335,-20){\large$\Sigma_0$}
			\end{overpic}
			\vspace{0.8cm}
			\caption{$\Sigma_1$ and $\Sigma_0$ in $C_n(\frac52)$. The dotted lines are dividing curves and the blue lines are $\gamma \times \{1\}$ and $\gamma \times\{0\}$}
			\label{fig:Sigma5}
		\end{center}
		\end{figure}

	First, suppose $m > 3$. Since $s \geq 2$, we can assume that $m > 2k-1$. In this case, the bypass cannot straddle the black dividing curve, see the right drawing of Figure~\ref{fig:Sigma4}. More precisely, there is no bypass passing through green, black, green dividing curves in order. Therefore, the only bypass that does not produce a boundary parallel arc is the one that straddles the green dividing curve. If the bypass straddles the green dividing curve, then attaching the bypass results in $\Sigma_{1/2}$ with $\Gamma_{1/2}=\{0,\frac{1}{n-2},\frac{1}{n-3}\}$ where $\mathrm{mult}(0)=3$, $\mathrm{mult}(\frac{1}{n-3})=m-4$ and $\mathrm{mult}(\frac{1}{n-2})=1$. We can convert $\Gamma_{1/2}$ into $\{0,1,\infty\}$ by acting on $\Sigma_{1/2}$ via $\phi_n^{-1}$.  Now we cut $C_n(s)$ along $\Sigma_{1/2}$ and obtain $\Sigma\times [0,1]$ again. We relable $\Sigma \times \{i\}$ as $\Sigma_i$ for $i = \{0,1\}$. Then $\Gamma_1 = \{0,1,\infty\}$ with $\mathrm{mult}(0) = 1$, $\mathrm{mult}(1) = m-4$, $\mathrm{mult}(\infty) = 3$. Take a closed curve $\gamma$ with slope $0$ and let $A_{\gamma} := \gamma\times[0,1]$ be a properly embedded annulus in $\Sigma\times[0,1]$. We perturb $A_{\gamma}$ so that it is convex with Legendrian boundary and intersects $\Gamma_0 \cup \Gamma_1$ minimally, see Figure~\ref{fig:Sigma5}. Since $|\gamma \times \{1\} \cap \Gamma_1| = m-1$ and $|\gamma \times \{0\} \cap \Gamma_0| = m-3$, we can find a bypass whose attaching arc lies on $\Sigma_1$ by Theorem~\ref{thm:Imbalance}. Notice that we attach this bypass from the back. It is easy to check that any possible bypass attachment results in an isotopic copy of $\Sigma$ that contains a boundary parallel dividing arc.    

	Now suppose $m = 3$. Since we assume $s \geq 2$, the only possible value of $k$ is $1$. In this case, there are exactly two possible bypass attachments on $\Sigma_0$ and both result in an isotopic copy of $\Sigma$ containing boundary parallel dividing arc, see Figure~\ref{fig:Sigma6}. 
\end{proof}

\begin{figure}[htbp]
	\begin{center}
		\begin{overpic}[tics=20]{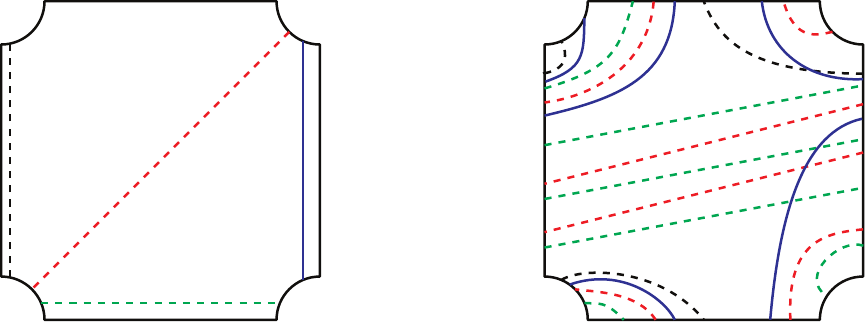}
			\put(72,-20){\large$\Sigma_1$}
			\put(335,-20){\large$\Sigma_0$}
		\end{overpic}
		\vspace{0.8cm}
		\caption{$\Sigma_1$ and $\Sigma_0$ in $C_n(3)$. The dotted lines are dividing curves and the blue lines are $\alpha \times \{1\}$ and $\alpha \times\{0\}$}
		\label{fig:Sigma6}
	\end{center}
\end{figure}

Now we consider the case $s < 0$.

\begin{proposition}\label{prop:negativeThickening}
	Suppose $s < 0$. Then there is an isotopic copy of $\Sigma$ in $C_n(s)$ such that $\Gamma_{\Sigma}$ contains a boundary parallel arc.
\end{proposition}

\begin{proof}
	Let $s=-\frac{m}{k}$ for a pair of relatively prime positive integers $m, k$. Then that there are $m$ properly embedded dividing arcs on $\Sigma$. We cut $C_n(s)$ along $\Sigma$ to obtain a genus-$2$ handlebody $\Sigma \times [0,1]$ and round edges to obtain a smooth convex boundary. Let $\Sigma_i$ be $\Sigma \times \{i\}$ for $i=0,1$ and $\Gamma_i$ the dividing set on $\Sigma_i$. Then $\Gamma_0 = \phi_n(\Gamma_1)$. After rounding the edges, the entire dividing set only contains closed curves, but we will keep calling the dividing curves that pass through $\bd\Sigma$ dividing arcs for convenience.  
    
  As in Proposition~\ref{prop:positiveThickening}, the dividing arcs in $\Gamma_i$ divides $\bd\Sigma_i$ into $2m$ intervals for $i=0,1$. The dividing curves on $\partial \Sigma \times [0,1]$ connect the $i$-th interval on $\Sigma_1$ to the $(i+2k-1)$-th interval (mod $2m$) on $\Sigma_0$. 

	Suppose there is no boundary parallel dividing arc in $\Gamma_i$ for $i=0,1$. By Proposition~\ref{prop:normalizationCn} and Proposition~\ref{prop:normalizationCn1}, we need to consider three cases.

	(1) {\it $m > 1$ and $\Gamma_1$ contains $m$ arcs and one closed curve with slope $\infty$}: As in Proposition~\ref{prop:positiveThickening}, let $\alpha$ be an arc on $\Sigma_1$ with slope $\infty$ that does not intersect $\Gamma_1$. Let $D_{\alpha}:= \alpha \times [0,1]$ be a compressing disk for $\Sigma \times [0,1]$. Perturb $D_{\alpha}$ so that its boundary does not intersect any dividing curve on $\partial \Sigma \times [0,1]$. This results in a shift of the basepoints of $\alpha$ on $\Sigma_0$ by $(2k-1)$ intervals following the positive orientation of $\partial \Sigma$. Thus $\bd D_{\alpha}$ only intersects the dividing curves in $\Gamma_0$. We perturb $D_{\alpha}$ further so that it is convex with Legendrian boundary. We can also assume that $\bd D_{\alpha}$ intersects $\Gamma_0$ minimally, so it will intersect the closed dividing curve exactly once. Also, this intersection point separates $\alpha \times \{0\}$ into two sides. On one side it intersects $|2k-1|$ dividing arcs, and on the other side it intersects $|m+(2k-1)|$ dividing arcs. Since the total number of intersection points is greater than $2$, there is a bypass in $D_{\alpha}$ by Theorem~\ref{thm:DiskImbalance}. The difference between the number of intersections is:
  \begin{equation*}
    d= |(m+2k-1) - (2k-1)| = |m|.
  \end{equation*}
  Since we assume $m > 1$, we have $d \neq 1$ and the dividing curves in $D_{\beta}$ cannot be nested as shown in the right drawing of Figure~\ref{fig:disk1}. This implies that there is always a bypass which does not straddle the closed dividing curve. After attaching this bypass, we obtain an isotopic copy of $\Sigma$ containing a boundary parallel dividing arc.

	(2) {\it $m \geq 3$ and $\Gamma_1$ contains arcs with slope $\{0,1,\infty\}$ where $\mathrm{mult}(0) = \mathrm{mult}(\infty) = 1$ and $\mathrm{mult}(1)=m-2$}: Let $\beta$ be an arc on $\Sigma_1$ with slopes $\infty$ that does not intersects $\Gamma_1$ and $D_{\beta}:= \beta \times [0,1]$ a compressing disk in $\Sigma \times [0,1]$. We perturb $D_{\beta}$ so that it is convex with Legendrian boundary and does not intersect dividing curves on $\bd \Sigma \times [0,1]$. Since $\beta \times \{0\}$ intersects more than four dividing curves on $\Gamma_0$, we can find a bypass in $D_{\beta}$ whose attaching arc lies on $\beta \times \{0\}$ by Theorem~\ref{thm:DiskImbalance}. Notice that there is no bypass passing through green, black, green dividing curves in order. Therefore, the only bypass that does not produce a boundary parallel arc is the one that straddles the green dividing curve. 
	
	Suppose $m > 3$. If the bypass straddles the green dividing curve, then attaching the bypass results in $\Sigma_{1/2}$ with $\Gamma_{1/2}=\{0,\frac{1}{n-2},\frac{1}{n-3}\}$ where $\mathrm{mult}(0)=3$, $\mathrm{mult}(\frac{1}{n-3})=m-4$ and $\mathrm{mult}(\frac{1}{n-2})=1$. We can convert $\Gamma_{1/2}$ into $\{0,1,\infty\}$ by acting on $\Sigma_{1/2}$ via $\phi_n^{-1}$.  Now we cut $C_n(s)$ along $\Sigma_{1/2}$ and obtain $\Sigma\times [0,1]$ again. We relable $\Sigma \times \{i\}$ as $\Sigma_i$ for $i = \{0,1\}$. Then $\Gamma_1 = \{0,1,\infty\}$ with $\mathrm{mult}(0) = 1$, $\mathrm{mult}(1) = m-4$, $\mathrm{mult}(\infty) = 3$. Take a closed curve $\gamma$ with slope $0$ and let $A_{\gamma} := \gamma\times[0,1]$ be a properly embedded annulus in $\Sigma\times[0,1]$. We perturb $A_{\gamma}$ so that it is convex with Legendrian boundary and intersects $\Gamma_0 \cup \Gamma_1$ minimally, see Figure~\ref{fig:Sigma5}. Since $|\gamma \times \{1\} \cap \Gamma_1| = m-1$ and $|\gamma \times \{0\} \cap \Gamma_0| = m-3$, we can find a bypass whose attaching arc lies on $\Sigma_1$ by Theorem~\ref{thm:Imbalance}. Notice that we attach this bypass from the back. It is easy to check that any possible bypass attachment results in an isotopic copy of $\Sigma$ that contains a boundary parallel dividing arc.    

	Now suppose $m = 3$. If the bypass straddles the green dividing curve, attaching the bypass results in an isotopic copy of $\Sigma$ containing one closed dividing curve and three dividing arcs with slope $0$. We can covert this into the dividing set containing one closed curve and three arcs with slope $\infty$ by acting  on $\Sigma$ via $\phi_n^{-1}$. We already dealt with this in Case (1).    

	\begin{figure}[htbp]
		\begin{center}
			\begin{overpic}[tics=20]{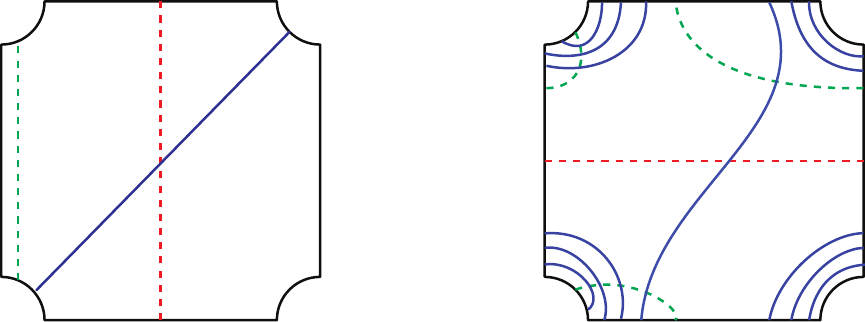}
				\put(72,-20){\large$\Sigma_1$}
				\put(335,-20){\large$\Sigma_0$}
			\end{overpic}
			\vspace{0.8cm}
			\caption{$\Sigma_1$ and $\Sigma_0$ in $C_n(-\frac12)$. The dotted lines are dividing curves and the blue lines are $\alpha \times \{1\}$ and $\alpha \times\{0\}$}
			\label{fig:Sigma7}
		\end{center}
	\end{figure}

	\begin{figure}[htbp]
		\begin{center}
			\begin{overpic}[tics=20]{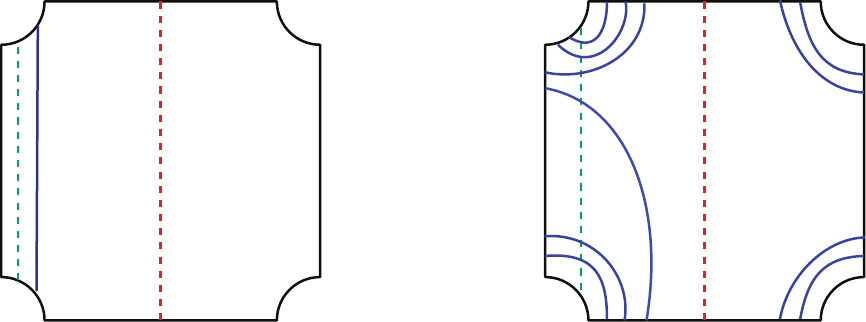}
				\put(72,-20){\large$\Sigma_1$}
				\put(335,-20){\large$\Sigma_{\frac12}$}
			\end{overpic}
			\vspace{0.8cm}
			\caption{$\Sigma_1$ and $\Sigma_{\frac12}$ in $C_n(-\frac12)$. The dotted lines are dividing curves and the blue lines are $\alpha \times \{1\}$ and $\alpha \times\{\frac12\}$}
			\label{fig:Sigma8}
		\end{center}
	\end{figure}
	
	(3) {\it $m = 1$ and there is one closed curve and one arc with slope $\infty$}: First, suppose $k > 1$. Then take an arc on $\Sigma$ with slope $1$ and let $D_{\alpha}$ be a disk in $\Sigma \times [0,1]$ as shown in Figure~\ref{fig:Sigma7}. The number of intersection between $D_{\alpha}$ and $\Gamma_0 \cup \Gamma_1$ is $4k+2$, so we can find a bypass in $D_{\alpha}$ by Theorem~\ref{thm:DiskImbalance}. If the bypass does not straddle the closed dividing curve, then attaching the bypass results in a boundary parallel dividing arc. If the bypass straddles the closed dividing curve, then attaching the bypass results in $\Sigma_{1/2}$ with the dividing set identical to $\Gamma_1$, see Figure~\ref{fig:Sigma8}. Now take an arc $\beta$ with slope $\infty$ and let $D_{\beta}$ be a compressing disk for $\Sigma \times [\frac12,1]$ as shown in Figure~\ref{fig:Sigma8}. Since the number of intersection between $D_{\beta}$ and $\Gamma_{1/2} \cup \Gamma_1$ is $4k-2$ and $k>1$, we can find a bypass in $D_{\beta}$. Here, any possible bypass does not straddle the closed dividing curve, so attaching the bypass will result in a boundary parallel dividing arc.  

	\begin{figure}[htbp]
		\begin{center}
			\begin{overpic}[tics=20]{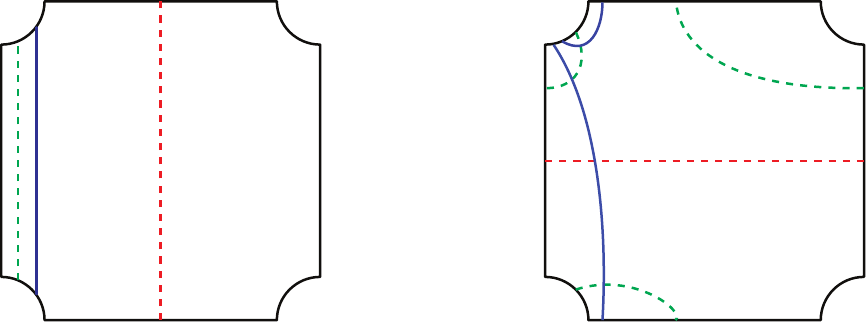}
				\put(72,-20){\large$\Sigma_1$}
				\put(335,-20){\large$\Sigma_0$}
			\end{overpic}
			\vspace{0.8cm}
			\caption{$\Sigma_1$ and $\Sigma_0$ in $C_n(-1)$. The dotted lines are dividing curves and the blue lines are $\alpha \times \{1\}$ and $\alpha \times\{0\}$}
			\label{fig:Sigma9}
		\end{center}
	\end{figure}

	\begin{figure}[htbp]
		\begin{center}
			\begin{overpic}[tics=20]{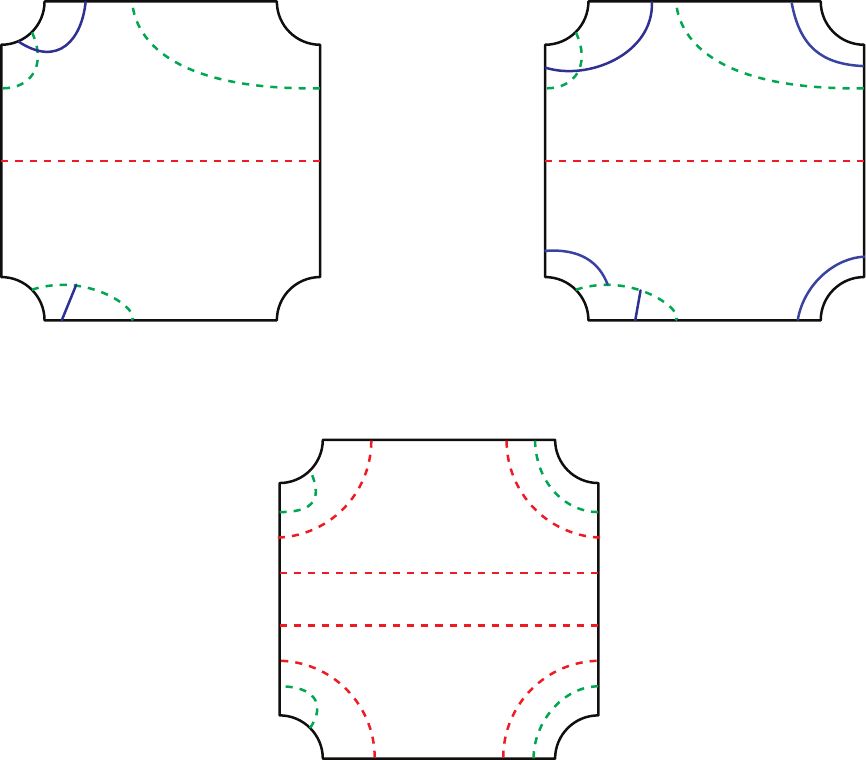}
				\put(72,195){\large$\Sigma_0$}
				\put(335,195){\large$\Sigma_0$}
				\put(204,-15){\large$\Sigma_{\frac12}$}
			\end{overpic}
			\vspace{0.8cm}
			\caption{Top left: $\Sigma_0$ in $C_n(-1)$, before the bypass rotation. Top right: $\Sigma_0$ in $C_n(-1)$, after the bypass rotation. Bottom: $\Sigma_{\frac12}$ in $C_n(-1)$, the result of the bypass attachment from the top right. The dotted lines are dividing curves and the blue lines are the attaching arcs of bypasses.}
			\label{fig:Sigma10}
		\end{center}
	\end{figure}
			
	Now suppose $k=1$. Take an arc on $\Sigma$ with slope $\infty$ and let $D_{\alpha} = \alpha \times [0,1]$ be a disk as shown in Figure~\ref{fig:Sigma9}. We call the closed dividing curve and the dividing arc on $\Sigma_0$ red and green dividing curves, respectively. Since the number of intersections between $D_{\beta}$ and $\Gamma_0 \cup \Gamma_1$ is $4$, there are two possible dividing sets for $D_{\beta}$. These two dividing sets provide two different bypasses. One bypass passes through red, green, green dividing curves in order. Attaching this bypass results in a boundary parallel dividing arc. The attaching arc of the other bypass does not lie on $\Sigma_0$. The attaching arc only passes through two dividing curves on $\Sigma_0$, see the top left drawing of Figure~\ref{fig:Sigma10}. By Theorem~\ref{thm:bypass-rotation}, we can rotate this bypass and obtain a new bypass such that its attaching arc lies on $\Sigma_0$, see the top right drawing of Figure~\ref{fig:Sigma10}. Attaching this bypass results in a boundary parallel dividing arc. 
\end{proof}

\subsection{Analyzing tight contact structures on \texorpdfstring{$C_n(\frac1k)$}{Cn(1/k)}  for \texorpdfstring{$k \leq 0$}{k <= 0}} \label{subsec:infty}
In this subsection, we will prove Lemma~\ref{lem:thickeningPositive}, Lemma~\ref{lem:tightInfty}, \ref{lem:thickeningNegative} and \ref{lem:tightReciprocal}. To do so, we need to prove Proposition~\ref{prop:normalizationCn1} first.

\begin{proof}[Proof of Proposition~\ref{prop:normalizationCn1}]
	Since there is only one dividing arc, we know $s = \frac1k$ for $k \in \mathbb{Z}$ (we consider $\infty = \frac10$). We cut $C_n(s)$ along $\Sigma$ to obtain a genus-$2$ handlebody $\Sigma \times [0,1]$. Let $\Sigma_i$ be $\Sigma \times \{i\}$ for $i=0,1$ and $\Gamma_i$ the dividing set on $\Sigma_i$. Then $\Gamma_0 = \phi_n(\Gamma_1)$. For $i=0,1$, the dividing arcs in $\Gamma_i$ divides $\bd\Sigma_i$ into two intervals, which we will label by $1$ and $2$. The dividing curves on $\partial \Sigma \times [0,1]$ connect the $i$-th interval on $\Sigma_1$ to the $(i-2k-1)$-th interval (mod $2$) on $\Sigma_0$. Notice that $i-2k-1 \equiv i+1 \pmod2$. By Proposition~\ref{prop:normalization1}, we need to consider two cases.

	(1) {\it $\Gamma_1$ consists of one closed curve and one arc with slope $r$}: Recall from Proposition~\ref{prop:normalizationCn} that we calculated the slopes $s^n_a$ and $s^n_r$ of fixed points of $\phi_n$ and showed 
	\[
		0 < s_a^n < 1/2 \;\text{ and }\; 2 < s_r^n < \infty.
	\] 
	Before we cut $C_n(s)$ along $\Sigma$, we can change $\Gamma$ by acting on $\Sigma$ via $(\phi_n)^m$ for $m \in \mathbb{Z}$. Since $\phi_n(\infty)=0$, we can assume $r$ satisfies either
	\begin{enumerate}
		\item[(a)] $-\infty < r \leq 0$, or
		\item[(b)] $s^n_a < r < 1 < s_r^n$.
	\end{enumerate}
	
	For the case (a), we know $0 < \phi_n(r) < s^n_a$ since $\phi_n(\infty)=0 < \phi_n(r)$ and $s_n^a$ is the slope of the attracting fixed point. Assume $r$ is not already $0$. Let $c$ be a closed curve on $\Sigma$ with dividing slope $\phi_n(r)$ and $A_c := c \times [0,1]$ be a properly embedded annulus in $\Sigma \times [0,1]$. We perturb $A_c$ so that it is convex with Legendrian boundary and intersects $\Gamma_0 \cup \Gamma_1$ minimally. Since $|c \times \{1\} \cap \Gamma_1| > 0$ and $|c \times \{0\} \cap \Gamma_0| = 0$, we can find a bypass whose attaching arc lies on $\Sigma_1$ by Theorem~\ref{thm:Imbalance}. We attach this bypass and keep track of dividing curves using Theorem~\ref{thm:bypass}. However, we should be careful according to Remark~\ref{rmk:convention}. First, we use the ordinary slope convention $\frac qp = \vects{p}{q}$ for $\Sigma$, so we should reverse the words ``clockwise'' and ``anticlockwise''. Also, since the bypass is attached from the back, we should reverse the words ``clockwise'' and ``anticlockwise'' again. Therefore, the slope of the diving curves changes in a clockwise direction in the Farey graph. To summarize, attaching the bypass results in an isotopic copy of $\Sigma$ with dividing slope $r'$ that is clockwise of $r$, anticlockwise of $\phi_n(r)$ and closest to $\phi_n(r)$ with an edge to $r$. Simply, $r < r'$. We cut $C_n(s)$ again along this $\Sigma$ and apply the same argument again. Repeat this procedure and we obtain an isotopic copy of $\Sigma$ with slope $0$ in the long run. By acting on $\Sigma$ via $\phi_n^{-1}$, we obtain $\Sigma$ with dividing slope $\infty = \phi_n^{-1}(0)$.     

	Now consider the case (b). In this case, we know $0 < \phi_n(r) < r < 1$ since $\phi_n(r)$ is closer to $s_a^n$, the slope of the attracting fixed point of $\phi_n$. Notice that the path in the Farey graph clockwise of $r$ and anticlockwise of $\phi_n(r)$ contains $1$ and $\infty$. Let $c$ be a closed curve on $\Sigma$ with dividing slope $\phi_n(r)$ and $A_c := c \times [0,1]$ be a properly embedded annulus in $\Sigma \times [0,1]$. By the same argument in the case (a), we can find a bypass  whose attaching arc lies on $\Sigma_1$. Again, according to Remark~\ref{rmk:convention} this bypass changes the slope of the diving curves in a clockwise direction in the Farey graph. To summarize, attaching the bypass results in an isotopic copy of $\Sigma$ with dividing slope $r'$ that is clockwise of $r$, anticlockwise of $\phi_n(r)$ and closest to $\phi_n(r)$ with an edge to $r$. We cut $C_n(s)$ again along this $\Sigma$ and apply the same argument again. Since the path in the Farey graph clockwise of $r$ and anticlockwise of $\phi_n(r)$ contains $1$, we can obtain an isotopic copy of $\Sigma$ with dividing slope $1$ by repeating this procedure. It is still true that $s^n_a < 1 < s_r^n$, so we can apply the same argument (attaching the bypass to $\Sigma_1$) again. Since $1$ and $\infty$ are connected by an edge in the Farey graph, we obtain $\Sigma$ with dividing slope $\infty$.

	\begin{figure}[htbp]
		\begin{center}
			\begin{overpic}[tics=20]{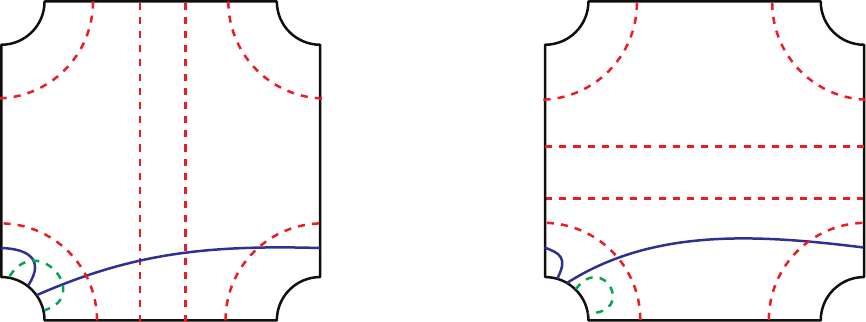}
				\put(72,-20){\large$\Sigma_1$}
				\put(335,-20){\large$\Sigma_0$}
			\end{overpic}
			\vspace{0.8cm}
			\caption{$\Sigma_1$ and $\Sigma_0$ in $C_n(\infty)$. The dotted lines are dividing curves and the blue lines are $\alpha \times \{1\}$ and $\alpha \times\{0\}$}
			\label{fig:Sigma11}
		\end{center}
	\end{figure}		

	\begin{figure}[htbp]
		\begin{center}
			\begin{overpic}[tics=20]{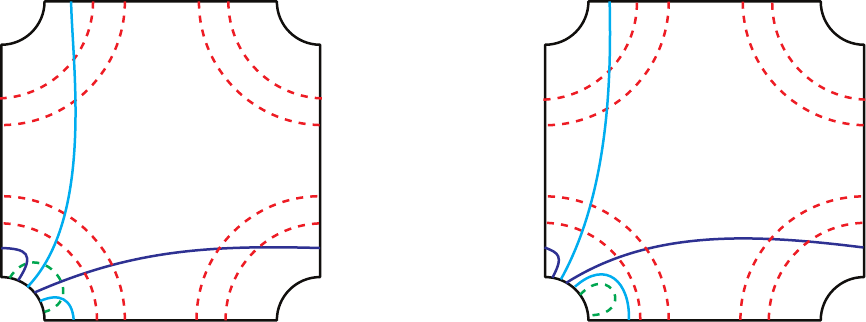}
				\put(72,-20){\large$\Sigma_1$}
				\put(335,-20){\large$\Sigma_0$}
			\end{overpic}
			\vspace{0.8cm}
			\caption{$\Sigma_1$ and $\Sigma_0$ in $C_n(\infty)$. The dotted lines are dividing curves, the blue lines are $\alpha \times \{1\}$ and $\alpha \times\{0\}$, and the sky blue lines are $\beta \times \{1\}$ and $\beta \times\{0\}$}
			\label{fig:Sigma12}
		\end{center}
	\end{figure}

	(2) {\it $\Gamma_1$ contains a boundary parallel arc}: In this case, $\Gamma_1$ may contain $j$ boundary parallel closed curves and $h$ essential closed curves, see Figure~\ref{fig:Sigma11} for example, and we need to get rid of them. Observe that $h$ is even. Let $r$ be the slope of the essential dividing curves and $\alpha$ be an arc on $\Sigma_1$ with slope $\phi(r)$ that intersects the dividing arc twice, the closed boundary parallel dividing curves $2j$ times, and the essential closed curves $h|\phi(r) \bigcdot r|$ times, see the left drawing of Figure~\ref{fig:Sigma11} for example. Observe that $|\phi(r) \bigcdot r| \geq 1$ since $\phi$ is pseudo-Anosov. Let $D_{\alpha}:= \alpha \times [0,1]$ be a compressing disk for $\Sigma \times [0,1]$. Perturb $D_{\alpha}$ so that its boundary does not intersect any dividing curve on $\partial \Sigma \times [0,1]$. This results in a shift of the basepoints of $\alpha$ on $\Sigma_0$ by $(2k+1)$ intervals following the positive orientation of $\partial \Sigma$. We perturb $D_{\alpha}$ further so that it is convex with Legendrian boundary and $\alpha \times \{0\}$ intersects boundary parallel closed dividing curves $2j$ times as shown in the right drawing of Figure~\ref{fig:Sigma11}. To summarize, $\alpha \times \{1\}$ intersects at least $(2j+h+2)$ dividing curves and $\alpha \times \{0\}$ intersects exactly $2j$ dividing curves. If $h \geq 2$, then we have $|\Gamma_1 \cap D_{\alpha}| - |\Gamma_0 \cap D_{\alpha}| \geq 4$, so we can find a bypass in $D_{\alpha}$ of which the attaching arc lies in $\alpha \times \{1\}$, see the right drawing of Figure~\ref{fig:disk2} for example. By attaching this bypass, we obtain an isotopic copy of $\Sigma$ with $2$ fewer dividing curves than the original one. We cut $C_n(s)$ along this new $\Sigma$ and repeat this until $h$ becomes $0$. 

	\begin{figure}[htbp]
		\begin{center}
			\begin{overpic}[tics=20]{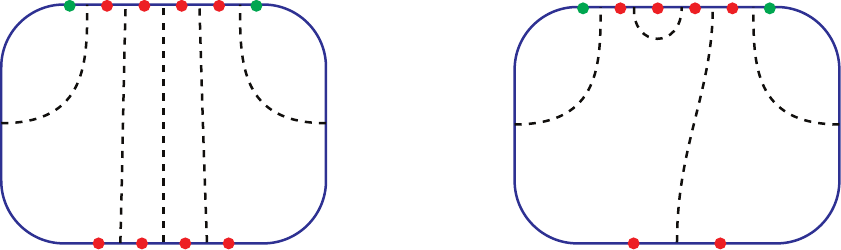}
			\end{overpic}
			\caption{Dividing curves in $D_{\alpha}$. Top sides: $\alpha\times\{1\}$. Bottom sides: $\alpha\times\{0\}$}
			\label{fig:disk2}
		\end{center}
	\end{figure}	

	Now we can assume $h=0$ and $j>0$. Here, we assume $s = \infty$. For other $s = \frac1k$, we attach a minimally twisting contact structure on $T = T^2\times [0,1]$ with dividing slopes $\infty$ and $\frac1k$ to $C_n(\frac1k)$ and obtain $C_n(\infty)$. Then remove $T$ at the end of the proof. Also, fix a sign of the bypass on $\Sigma_1$ to be negative. The argument will be the same for the positive case. Now we have 
	\[
		|\Gamma_1 \cap D_{\alpha}| - |\Gamma_0 \cap D_{\alpha}| = (2j+h+2) - 2j  = 2,
	\]
	so there is a dividing set on $D_{\alpha}$ that does not provide a bypass of which the attaching arc lies in $\alpha\times\{1\}$ as shown in the left drawing of Figure~\ref{fig:disk2}. Observe that all other dividing sets on $D_{\alpha}$ (up to isotopy) will give a bypass of which the attaching arc lies in $\alpha \times \{1\}$. Now take an arc $\beta$ on $\Sigma_1$ such that $\Sigma \setminus (\alpha\cup \beta)$ is a disk and $\beta$ intersects $\Gamma_1$ $(2j+2)$ times, see the left drawing of Figure~\ref{fig:Sigma12} for example ($j=2$). Let $D_{\beta} := \beta\times [0,1]$ be a compressing disk for $\Sigma\times[0,1]$. If we cut $\Sigma\times[0,1]$ along $D_{\alpha}$ and $D_{\beta}$ and round the edges, then we obtain a $3$-ball $B^3$ with convex boundary. Since there is a unique tight contact structure on a $3$-ball with a fixed characteristic foliation, tight contact structures on $\Sigma\times[0.1]$ are completely determined by the dividing sets on $D_{\alpha}$ and $D_{\beta}$. As discussed above, all dividing sets on $D_{\alpha}$ (\emph{resp.} $D_{\beta}$) contain a bypass of which the attaching arc lies in $\alpha\times\{1\}$ (\emph{resp.} $\beta\times\{1\}$) except for the one shown in the left drawing of Figure~\ref{fig:disk2}. Therefore, in every tight contact structure on $\Sigma \times [0,1]$, except for (possibly) one, there is a bypass for $\Sigma_1$ of which the attaching arc lies in $\alpha\times\{1\}$ or $\beta\times\{1\}$. As we observed above, attaching this bypass reduces the number of dividing curves. Denote the contact structure that might not contain such a bypass by $\xi$. By the Thurston--Bennequin inequality for convex surfaces \cite{Giroux:classification2}, there is no bypass that decreases the number of dividing curves in an $I$-invariant neighborhood of a convex surface. Since the dividing curves on $\bd \Sigma \times [0,1]$ are vertical, the contact structure $\xi$ should be contactomorphic to an $I$-invariant neighborhood of $\Sigma_1$. 

	\begin{figure}[htbp]
	\begin{center}
		\begin{overpic}[tics=20]{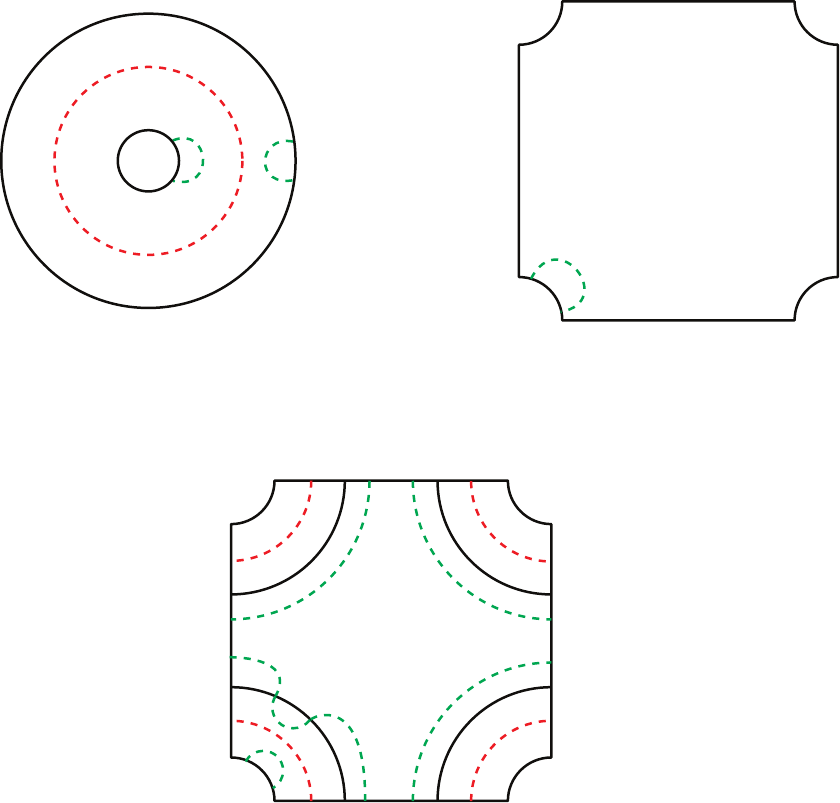}
		\end{overpic}
		\caption{Top left: annulus in $T^2 \times [0,1]$. Top right: $\Sigma$ in $C_n(\infty)$. Bottom: $\Sigma$ in $C_n(\infty)$ after gluing $T^2 \times [0,1]$. The dotted lines are dividing curves.} 
		\label{fig:torsion}
	\end{center}
	\end{figure}
	
	Now consider a convex Giroux $\frac{j}2$-torsion layer $(T^2 \times [0,1],\xi_{j/2})$ with boundary slope $\infty$. Since $T^2 \times [0,1] = A\times S^1$, we cut $T^2 \times [0,1]$ along $A$ and obtain $A \times [0,1]$. It is well known ({\it c.f.} \cite{Honda:classification2,Mathews:torsion}) that there is a convex annulus $A$ in $\xi_{j/2}$ with a boundary parallel dividing arc on each boundary component and $(j-1)$ closed boundary parallel dividing curves, and $A\times [0,1]$ is an $I$-invariant neighborhood, see the top left drawing of Figure~\ref{fig:torsion} for Giroux $1$-torsion ($j=2$). We also consider a tight contact structure $\xi_0$ on $\Sigma\times [0,1]$, an $I$-invariant neighborhood of a convex surface $\Sigma$ with a single dividing arc without any other dividing curves, see the top right drawing of Figure~\ref{fig:torsion}. With appropriate choices of the signs of the bypasses on $\Sigma$ and $A$, we obtain an $I$-invariant neighborhood of a convex surface with one dividing arc and $j$ boundary parallel closed curves by gluing two contact structures $\xi_0$ and $\xi_{j/2}$, see the bottom drawing of Figure~\ref{fig:torsion}. Therefore, $\xi$ is contactomorphic to $\xi_{j/2} \cup \xi_0$. We glue the top and bottom of ($\Sigma\times[0,1],\xi$) using $\phi_n$ and obtain $C_n(s)$. Then it clearly contains Giroux $\frac{j}{2}$-torsion. Thus we can exclude this case, and in any tight contact structure on $C_n(s)$ without boundary parallel half Giroux torsion, we can find a bypass for $\Sigma_1$ of which the attaching arc lies on $\alpha\times\{1\}$ or $\beta\times\{1\}$. By attaching this bypass we obtain an isotopic copy of $\Sigma$ that contains $2$ fewer dividing curves than $\Sigma_1$. We cut $C_n(s)$ along this new $\Sigma$ and repeat the argument until we obtain $\Sigma$ without closed dividing curves.         
\end{proof}

According to Honda~\cite{Honda:classification1}, there is a unique tight contact structure on $L(n,n-1)$ up to isotopy. We denote it by $\xi_{std}$. As explained in Section~\ref{subsec:decomp}, $K_2$ can be considered as a knot in $L(n,n-1)$. Also, as shown in Figure~\ref{fig:legendrian1}, there is a Legendrian realization $L$ of $K_2$ in $(L(n,n-1),\xi_{std})$ with $\tb = 1$ and $\rot=0$ (which can be computed by using the formula from \cite[Lemma~6.6]{LOSS:loss}). We just have proved the following lemma. 

\begin{lemma}\label{lem:K2leg}
	There is a Legendrian representative $L$ of $K_2$ in $(L(n,n-1),\xi_{std})$ such that $\tb(L) = 1$ and $\rot(L)=0$. \qed
\end{lemma}

Now we are ready to prove the lemmas in Section~\ref{subsec:decomp}. We begin with Lemma~\ref{lem:thickeningPositive}

\begin{proof}[Proof of Lemma~\ref{lem:thickeningPositive}]
	Since $r \in \mathcal{R}_+$ and $s \in \mathcal{S}(r)$, we know $s > 2$ and $s \notin (4,5]$. By Proposition~\ref{prop:positiveThickening}, we can find a bypass for $-\bd C_n(s)$ with slope $0$. By Theorem~\ref{thm:bypass}, we can thicken $C_n(s)$ to $C_n(s')$ such that $s' > 2$ and $s' \notin (4,5]$. If $s' \neq \infty$, we can find a bypass again by Proposition~\ref{prop:positiveThickening} and repeat this until we obtain $C_n(\infty)$.
\end{proof}

\begin{proof}[Proof of Lemma~\ref{lem:tightInfty}]
	Again, after cutting $C_n(\infty)$ along $\Sigma$ and rounding the edges, we obtain a genus-$2$ handlebody with convex boundary. Let $\Sigma_i$ be $\Sigma \times \{i\}$ for $i=0,1$ and $\Gamma_i$ the dividing set on $\Sigma_i$. Then $\Gamma_0 = \phi_n(\Gamma_1)$. The dividing arc in each $\Gamma_i$ divides $\bd\Sigma_i$ into two intervals, which we will label by $1$ and $2$. The dividing curves on $\partial \Sigma \times [0,1]$ connect the $i$-th interval on $\Sigma_1$ to the $(i-1)$-th interval (mod $2$) on $\Sigma_0$. By Proposition~\ref{prop:normalizationCn1}, we need to consider two cases. 

	(1) {\it $\Gamma_1$ consists of a boundary-parallel arc without any other dividing curves}: First, we fix the sign of the bypass on $\Sigma_1$ to be negative. This will determine the signs of the regions of entire $\bd(\Sigma\times [0,1])$. Also, the relative Euler class evaluated on $\Sigma_1$ is $-2$. We observe that the dividing set on $\Sigma\times[0,1]$ is a single closed curve parallel to $\bd\Sigma$. Thus we can choose two convex compressing disks for the handlebody of which the Legendrian boundaries intersect the dividing curve exactly twice, see Figure~\ref{fig:Sigma13}. Thus there is a unique (possibly) tight contact structure on this handlebody, and thus a unique (possibly) tight contact structure on $C_n(\infty)$, which we denote by $\xi^-_{\infty}$. Now we change the sign of the bypass on $\Sigma_1$ to be positive. Then the relative Euler class evaluated on $\Sigma_1$ is $2$. By repeating the same argument, we can show that there is a unique (possibly) tight contact structure on $C_n(\infty)$, which we denote by $\xi^+_{\infty}$.   

	It follows that any tight contact structure on $C_n(\infty)$ whose relative Euler class evaluated on $\Sigma$ is $\pm2$ without boundary parallel half Giroux torsion is actually isotopic to $\xi^{\pm}_{\infty}$, respectively. Now Let $\xi^-_{-1}$ be a contact structure on $C_n(-1)$, the complement of a standard neighborhood of $S_-(S_-(L))$, where $S_-$ is the negative stabilization operator and $L$ is a Legendrian representative of $K_2$ in $(L(n,n-1),\xi_{std})$ with $\tb(L)=1$ and $\rot(L)=0$ according to Lemma~\ref{lem:K2leg}. Also, denote by $\xi^-_B$ the contact structure on a negative basic slice $B_-(\infty,-1)$. Honda \cite{Honda:classification1} showed that any basic slice is universally tight, so $\xi_B^-$ is universally tight. Also $\xi_{-1}^-$ is universally tight according to Colin's gluing theorem \cite{Colin:gluing} since the gluing surfaces only contain boundary parallel dividing arcs (this condition is called \dfn{well-groomed}). If we glue $\xi_0^-$ and  $\xi^-_B$ along the boundary torus $T_{-1}$ with dividing slope $-1$, then we obtain a tight contact structure on $C_n(\infty)$ since $T_{-1}$ is rotative and thus we can apply the gluing theorem for universally tight contact structures along a rotative torus~\cite{Colin:gluing,HKM:decomposition}. Clearly the relative Euler class evaluated on $\Sigma$ in $\xi^{-}_{0} \cup \xi^-_B$ is $-2$, so the tight contact structure $\xi^{-}_{0} \cup \xi^-_B$ is isotopic to $\xi^-_{\infty}$. By the same argument, $\xi^{+}_{0} \cup \xi^+_B$ is isotopic to $\xi^+_{\infty}$. Therefore, $(C_n(\infty),\xi^{\pm}_{\infty})$ in fact thickens to $C_n(1)$.

	\begin{figure}[htbp]
	\begin{center}
		\begin{overpic}[tics=20]{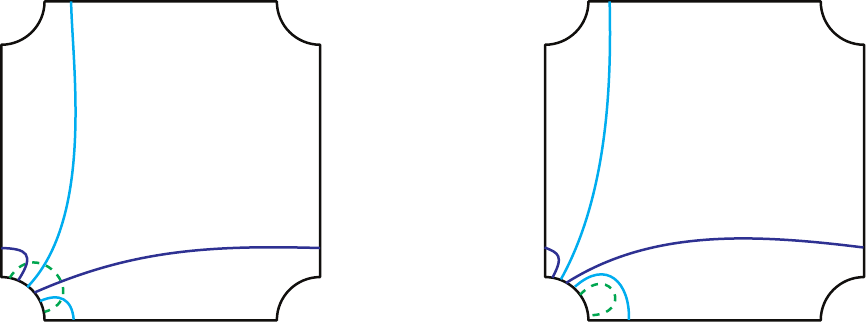}
			\put(72,-20){\large$\Sigma_1$}
			\put(335,-20){\large$\Sigma_0$}
		\end{overpic}
		\vspace{0.8cm}
		\caption{$\Sigma_1$ and $\Sigma_0$ in $C_n(\infty)$. The dotted lines are dividing curves, the blue lines are $\alpha \times \{1\}$ and $\alpha \times\{0\}$, and the sky blue lines are $\beta \times \{1\}$ and $\beta \times\{0\}$}
		\label{fig:Sigma13}
	\end{center}
	\end{figure}

	\begin{figure}[htbp]
	\begin{center}
		\begin{overpic}[tics=20]{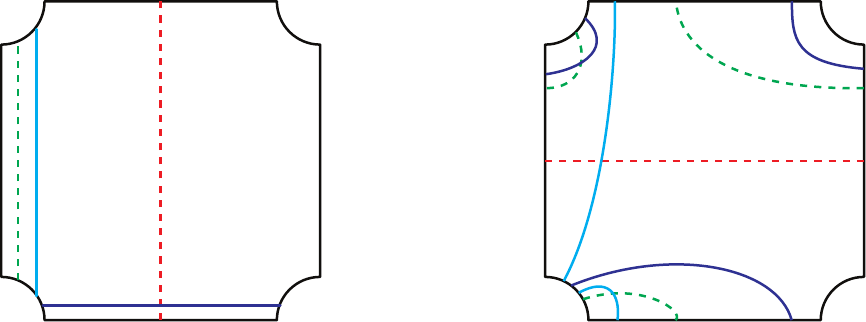}
			\put(72,-20){\large$\Sigma_1$}
			\put(335,-20){\large$\Sigma_0$}
		\end{overpic}
		\vspace{0.8cm}
		\caption{$\Sigma_1$ and $\Sigma_0$ in $C_n(\infty)$. The dotted lines are dividing curves, the blue lines are $\alpha \times \{1\}$ and $\alpha \times\{0\}$, and the sky blue lines are $\beta \times \{1\}$ and $\beta \times\{0\}$}
		\label{fig:Sigma14}
	\end{center}
	\end{figure}

	(2) {\it $\Gamma_1$ contains one arc and one closed curve with slope $\infty$}: In this case, the relative Euler class evaluated on $\Sigma$ is $0$. Fix the signs of the regions $\Sigma_1 \setminus \Gamma_1$. Take arcs $\alpha$ and $\beta$ on $\Sigma_1$ with slopes $\infty$ and $0$ as shown in the left drawing of Figure~\ref{fig:Sigma14}. Take compressing disks $D_{\alpha} = \alpha \times [0,1]$ and $D_{\beta}=\beta\times [0,1]$ for $\Sigma\times[0,1]$. Perturb the disks so that their boundaries do not intersect any dividing curve on $\partial \Sigma \times [0,1]$. This results in a shift of the basepoints of $\alpha$ and $\beta$ on $\Sigma_0$ by $-1$ intervals following the positive orientation of $\partial \Sigma$. Then $\bd D_{\alpha}$ and $\bd D_{\beta}$ intersect the dividing curves on $\bd(\Sigma \times [0,1])$ exactly twice, respectively, as shown in Figure~\ref{fig:Sigma14}. Thus there is a unique (possibly) tight contact structure on this handlebody, and thus a unique (possibly) tight contact structure on $C_n(\infty)$. Now we change the sign of the regions $\Sigma_1 \setminus \Gamma_1$. By repeating the same argument, we can show that there is a unique (possibly) tight contact structure on $C_n(\infty)$. In total, there are at most two (possibly) tight contact structures on $C(\infty)$.	
\end{proof}

\begin{remark}
	Since the dividing curves on $\Sigma$ may have some boundary twisting for the case (2), there could be infinitely many tight contact structures on $C_n(\infty)$ if we fix the boundary pointwise.
\end{remark}

Now we will prove Lemma~\ref{lem:thickeningNegative} and~\ref{lem:tightReciprocal}. The argument will be similar to the proof of \cite[Lemma~3.1 and~3.4]{CM:figure-eight}. However, since $\phi_n$ is right veering, whereas the monodromy of the figure-eight knot is not, the maximal slope to which $C_n(s)$ can thicken is different.

\begin{proof}[Proof of Lemma~\ref{lem:thickeningNegative} and~\ref{lem:tightReciprocal}]
	Since $r \in \mathbb{Q}_-$, $s \in \mathcal{S}(r)$ and $s\neq 0$, we know $s < 0$. Let $k \in \mathbb{N}$ such that $-\frac1{k-1} < s \leq -\frac1k$. Suppose $s \neq -\frac1k$. Then by Proposition~\ref{prop:negativeThickening} we can find a bypass for $-\bd C_n(s)$ with slope $0$ and we can thicken $C_n(s)$ to $C_n(s')$ where $s < s' \leq -\frac1k$. Thus we can find a bypass again by Proposition~\ref{prop:negativeThickening} and repeat this until we obtain $C_n(-\frac1k)$. Let $\Gamma_{\Sigma}$ be a dividing set for $\Sigma$ in $C_n(-\frac1k)$. By Proposition~\ref{prop:normalizationCn1}, we need to consider two cases.

	(1) {\it $\Gamma_{\Sigma}$ contains one boundary parallel arc without any other dividing curves}: First, we fix the sign of the bypass on $\Sigma$ to be negative. Then the relative Euler class evaluated on $\Sigma$ is $-2$. After cutting $C_n(-\frac 1k)$ along $\Sigma$ and rounding the edges, we obtain $\Sigma\times[0,1]$ with a smooth convex boundary. Also, the dividing set consists of a single closed curve parallel to $\bd\Sigma$. By choosing convex compressing disks for this handlebody whose Legendrian boundaries intersect the dividing curve exactly twice as shown in Figure~\ref{fig:Sigma13}, we see that there is a unique (possibly) tight contact structure on this handlebody, and thus a unique (possibly) tight contact structure on $C_n(-\frac 1k)$, which we denote by $\xi^-_{-1/k}$.
	
	It follows that any tight contact structure on $C_n(-\frac 1k)$ whose relative Euler class evaluated on $\Sigma$ is $-2$ without boundary parallel half Giroux torsion is actually isotopic to $\xi^-_{-1/k}$. Hence the complement of a standard neighborhood of $S_{-}(S_{-}(L))$ in $(L(n,n-1),\xi_{std})$ is isotopic to $\xi^-_{-1}$, where $S_{-}$ is the negative stabilization operator and $L$ is a Legendrian representative of $K_2$ from Lemma~\ref{lem:K2leg}.  Therefore, $(C_n(-1),\xi^-_{-1})$ in fact thickens to $C_n(1)$, and $C_n(-1)=T_- \cup C(\infty)$, where $T_- = B_-(-1,0) \cup B_-(0,1)$ is a union of two negative basic slices with dividing slopes $s_0=-1$ and $s_1=1$. We can further decompose $T_-$ into $T^k_- \cup B_-(-\frac1k, 0) \cup B_-(0,1)$, where $T^k_-$ is a minimally twisting contact $T^2\times[0,1]$ with dividing slopes $s_0=-1$ and $s_1=-\frac 1k$. Hence we obtain a tight contact structure on $C_n(-\frac 1k) = B_-(-\frac1k, 0) \cup B_-(0,1) \cup C(1)$. It is straightforward to check that the relative Euler class evaluated on $\Sigma$ in this $C_n(-\frac1k)$ is $-2$, so this contact structure is isotopic to $\xi^-_{-1/k}$. Therefore, $(C_n(-\frac 1k), \xi^-_{-1/k})$ also thickens to $C_n(1)$ for all $k \in \mathbb{N}$.
	
	Next, change the signs of the regions of $\Sigma$ so that the sign of the bypass is positive. In this case, the relative Euler class evaluated on $\Sigma$ is $2$. After repeating the above argument, we get a unique tight contact structure under these conditions, which we denote by $\xi^+_{-1/k}$.  Similarly, we can build the tight contact structures isotopic to $(C_n(-\frac 1k),\xi^+_{-1/k})$ that thicken to $C_n(1)$.
	
	The above arguments show that these tight contact structures on $C_n(-\frac 1k)$ are determined by the sign of the bypass on $\Sigma$, which is in turn determined by the sign on the stacks of basic slices $B_{\pm}(-\frac1k,0) \cup B_{\pm}(0,1)$, and are thus independent of the tight contact structure on $C_n(1)$.

	(2) {\it $\Gamma_{\Sigma}$ contains one closed curve and one arc with slope $\infty$}: In this case, the relative Euler class evaluated on $\Sigma$ is $0$. Fix the signs of the regions $\Sigma \setminus \Gamma_{\Sigma}$. We showed in the proof of Proposition~\ref{prop:negativeThickening} that there exists an isotopic copy of $\Sigma$ in $C_n(-\frac 1k)$ with one boundary parallel dividing arc, one closed boundary parallel dividing curve and perhaps two closed essential curves, see the bottom drawing of Figure~\ref{fig:Sigma10} for example. As in the proof of Proposition~\ref{prop:normalizationCn1}, we can remove the closed essential dividing curves and hence we can assume $\Gamma_{\Sigma}$ consists of one boundary parallel arc and one boundary parallel closed curve. This implies that the dividing set of $\bd(\Sigma\times[0,1])$ consists of exactly three dividing curves parallel to $\bd\Sigma$. By Cofer \cite[Section~3]{Cofer:handlebody}, there exist two tight contact structures on this $\Sigma\times[0,1]$ (with a fixed characteristic foliation). Thus there exist at most two tight contact structures on $C_n(-\frac1k)$.

	According to the proof of Proposition~\ref{prop:negativeThickening}, it is clear that one of these contact structures contains a boundary parallel half Giroux torsion layer. By repeating the same argument with the opposite signs on the regions on $\Sigma \setminus \Gamma_{\Sigma}$, we see that there are at most two tight contact structures on $C_n(-\frac 1k)$ without boundary parallel half Giroux torsion when the relative Euler class evaluated on $\Sigma$ is $0$, and that these contact structures are completely determined by the sign of the bypass on $\Sigma$. Denote these contact structures by $(C_n(-\frac 1k),\xi'^{\pm}_{-1/k})$.

	It follows that any tight contact structure on $C_n(-\frac 1k)$ whose relative Euler class evaluated on $\Sigma$ is $0$ without boundary parallel half Giroux torsion is actually isotopic (fixing boundary setwise) to either $\xi'^{+}_{-1/k}$ or $\xi'^{-}_{-1/k}$. Thus according to Lemma~\ref{lem:K2leg}, the complement of a standard neighborhood of $S_{+}(S_{-}(L))$ (\emph{resp.} $S_-(S_+(L))$) in $(L(n,n-1),\xi_{std})$ is isotopic to $\xi'^+_{-1}$ (\emph{resp.} $\xi'^-_{-1}$). Therefore, $(C_n(-1),\xi'^{\pm}_{-1})$ in fact thickens to $C_n(1)$, and $C_n(-1)=T'_{\pm} \cup C_n(1)$, where $T'_{\pm} = B_{\pm}(-1,0) \cup B_{\mp}(0,1)$ is a union of two negative basic slices with dividing slopes $s_0=-1$ and $s_1=1$. We further decompose $T'_{\pm}$ into $T'^k_{\pm} \cup B_{\pm}(-\frac1k, 0) \cup B_{\mp}(0,1)$, where $T'^k_{\pm}$ is a minimally twisting contact $T^2\times[0,1]$ with dividing slopes $s_0=-1$ and $s_1=-\frac 1k$. Hence we obtain tight contact structures on $C_n(-\frac 1k) = B_{\pm}(-\frac1k, 0) \cup B_{\mp}(0,1) \cup C_n(1)$. It is straightforward to check that the relative Euler class evaluated on $\Sigma$ is $0$ and the sign of the bypass is $\pm$, so these contact structures are isotopic (fixing boundary setwise) to $\xi'^{\pm}_{-1/k}$. Therefore, $(C_n(-\frac 1k), \xi'^{\pm}_{-1/k})$ also thickens to $C_n(1)$ for all $k \in \mathbb{N}$. 
	
	The above arguments show that these tight contact structures on $C_n(-\frac 1k)$ are determined by the sign of the bypass on $\Sigma$, which is in turn determined by the sign on the stacks of basic slices $B_{\pm}(-\frac1k,0) \cup B_{\mp}(0,1)$, and are thus independent of the tight contact structure on $C_n(1)$.
\end{proof}

\subsection{Tight contact structures on a solid torus}\label{subsec:solid}
In this subsection, we prove Lemma~\ref{lem:solid-torus} by counting the number of tight contact structures on $N_r(\infty)$ and $N_r(1)$.

\begin{proof}[Proof of Lemma~\ref{lem:solid-torus}]
	By Proposition~\ref{prop:solid-torus}, there exist $\Phi(r)$ tight contact structures on $N_r(\infty)$. 
	
	Now we consider $N_r(1)$. Let $r = \frac pq$ and $r \neq 1$. We can change the coordinates of $\bd N_r(1)$ by the matrix  
	\[
		\matrixp{0}{1}{-1}{1}.
	\]
	Then the dividing slope of $N_r(1)$ will become $\infty$ and the meridional slope will become $\frac{q}{q-p} = \frac{1}{1-r}$ (recall that we use the slope convention $\vects{p}{q} = \frac pq$ for a solid torus, see Remark~\ref{rmk:convention}). Therefore, there are $\Phi(\frac{1}{1-r}) = \Psi(r)$ tight contact structures on $N_r(1)$.
\end{proof}

\section{Fillability and virtual overtwistedness}

In this section, we will prove Theorem~\ref{thm:fillability} and Theorem~\ref{thm:virtually-ot}. We begin with the fillability of tight contact structures on $M(n,r)$.

\begin{proof}[Proof of Theorem~\ref{thm:fillability}]
	First, we consider the tight contact structures counted by $\Phi$. In \cite{EMT:torus}, tight contact structures on $S^3_n(T_{2,3})$ are classified and it is shown that every tight contact structure on $S^3_n(T_{2,3})$ is Stein fillable for $n \geq 5$. Recall from Section~\ref{sec:lower-bounds} that the tight contact structures counted by $\Phi$ can be constructed from the surgery diagram shown in Figure~\ref{fig:legendrian2}, which can be considered as a result of Legendrian surgeries on the knots in some tight contact structure on $S^3_{n}(T_{2,3})$. Since Legendrian surgery preserves Stein fillability, all tight contact structures counted by $\Phi$ for $n \geq 5$ are Stein fillable.
	
	Next we consider the tight contact structures counted by $\Psi$. Observe that there is a unique tight contact structure on $M(n,2)$ counted by $\Psi$ and denote it by $(M(n,2),\xi_2)$. We claim that $(M(n,2),\xi_2)$ is Stein fillable and show that any tight contact structure on $M(n,r)$ for $r > 2$ counted by $\Psi$ is obtained by some negative contact surgery on some Legendrian knot in $(M(n,2),\xi_2)$. Since negative contact surgery preserves Stein fillability, these will complete the proof.
	
	Lemma~\ref{lem:tightReciprocal} and the proof of Theorem~\ref{thm:upperBounds} implies that any contact structure on $M(n,r)$ for $n \geq 2$ and $r \geq 2$ counted by $\Psi$ can be obtained by a positive contact surgery on a Legendrian push-off of the binding of an open book decomposition of $(L(n,n-1),\xi_{std})$, the unique tight contact structure on $L(n,n-1)$. This in fact implies that any tight contact structure on $M(n,r)$ for $r > 2$ counted by $\Psi$ can be obtained by some negative contact surgery on some Legendrian knot in $(M(n,2),\xi_2)$.
 
 	To be more precise, the open book is shown in the left drawing of Figure~\ref{fig:Psi-open-book} where the monodromy of the open book is $\phi_n = D_{\alpha} D_{\beta} D_{\alpha}^{n-1}$, where $D_{\alpha}$ is a positive Dehn twist about $\alpha$. This positive contact surgery is equivalent to an inadmissible transverse $2$-surgery on the binding of the open book followed by a negative contact surgery on a Legendrian push-off of the surgery dual knot (\dfn{c.f.} \cite{Conway:transverse}). Also this inadmissible transverse $2$-surgery yields $(M(n,2),\xi_2)$. 
	
	By applying an algorithm in \cite{Conway:transverse} that converts the result of an inadmissible transverse surgery into an open book decomposition, we obtain an open book decomposition for $(M(n,2),\xi_2)$ as shown in the right drawing of Figure~\ref{fig:Psi-open-book} where the monodromy is 
	\[
		\widetilde{\phi}_n = D_{\alpha} D_{\beta} D_{\alpha}^{n-1}D_{\gamma}^{-1}\Delta 
	\] 
	and $\Delta$ is the composition of positive Dehn twists about each boundary component. Korkmaz and Ozbagci \cite{KO:factorizations} provided factorizations of $\Delta$. In particular, they showed  
	\[
		\Delta = D_{\alpha}^{-2} D_{\gamma} D_{\delta} D_{\sigma} = D_{\gamma} D_{\alpha}^{-2} D_{\delta} D_{\sigma}
	\]
	where $\delta$, $\sigma$ are some simple closed curves on the page. The second equality holds since $\alpha$ and $\gamma$ are disjoint. Now we have
	\begin{align*}
		\widetilde{\phi}_n &= D_{\alpha} D_{\beta} D_{\alpha}^{n-1} D_{\gamma}^{-1} \Delta\\ 
		&= D_{\alpha} D_{\beta} D_{\alpha}^{n-1} D_{\gamma}^{-1}D_{\gamma} D_{\alpha}^{-2} D_{\delta} D_{\sigma}\\
		&= D_{\alpha} D_{\beta} D_{\alpha}^{n-3} D_{\delta} D_{\sigma}.
	\end{align*} 
	Since $n \geq 5$, there is a factorization of $\phi \Delta$ which is a product of positive Dehn twists. Thus $(M(n,2),\xi_2)$ is Stein fillable.
\end{proof}

\begin{figure}[htbp]
	\begin{center}
		\vspace{0.5cm}
		\begin{overpic}[tics=20]{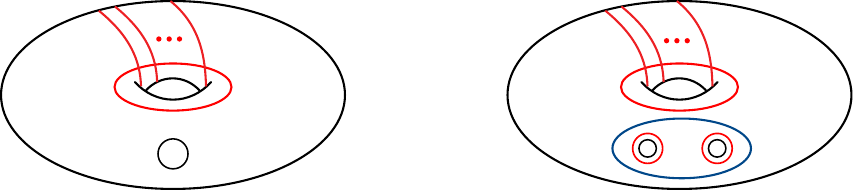}
		\put(45,70){\small$\alpha$}
		\put(120,48){\small$\beta$}

		\put(285,18){\small$\gamma$}
		\end{overpic}
		\caption{Left: an open book decomposition for $(L(n,n-1),\xi_{std})$. Right: an open book decomposition for $(M(n,2), \xi_2)$ counted by $\Psi$. The red curves represent positive Dehn twists and the blue curve represents a negative Dehn twist.} 
		\label{fig:Psi-open-book}
	\end{center}
\end{figure}

Lastly, we will consider the virtual overtwistedness of tight contact structures on $M(n,r)$.

\begin{proof}[Proof of Theorem~\ref{thm:virtually-ot}]
	First, as shown in Figure~\ref{fig:Weeks1}, $M(n,r)$ can be obtained by $r$-surgery on $K_2$ in $L(n,n-1)$ and the surgery dual knot $K_*$ represents a nontrivial element in $\pi_1(M(n,r))$. Thus there is a cover of $M(n,r)$ that unwraps $K_*$. Let $N$ be a neighborhood of $K_*$. By Honda \cite[Proposition~5.1.(2)]{Honda:classification1}, if $\xi$ is a virtually overtwisted contact structure on $N$, then any cover of $(N,\xi)$ contains an overtwisted disk. Thus if $(N,\xi)$ is virtually overtwisted, then there is an overtwisted cover of $(N,\xi)$ that embeds into some cover of $M(n,r)$. 	
	
	Now according to the proof of Theorem~\ref{thm:upperBounds}, the tight contact structures counted by $\Psi$ can be decomposed into $C(1)$ and $N_r(1)$. By Honda \cite[Proposition~5.1.(2)]{Honda:classification1}, there are at most two universally tight contact structures on $N_r(1)$. Thus for any $r \in \mathcal{R}_+ \cup \mathbb{Q}_-$, there are at most two universally tight contact structures on $M(n,r)$ counted by $\Psi$ and all others are virtually overtwisted. Also, the tight contact structures counted by $\Phi$ can be decomposed into $C(\infty)$ and $N_r(\infty)$. By Honda \cite[Proposition~5.1.(2)]{Honda:classification1} again, there are at most two universally tight contact structures on $N_r(\infty)$. Since there are two tight contact structures on $C_n(\infty)$ counted by $\Phi$, there are at most four universally tight contact structures on $M(n,r)$ for $r \in \mathcal{R}_+$ counted by $\Phi$, and all others are virtually overtwisted.
\end{proof}

\bibliography{references}{}
\bibliographystyle{plain}
\end{document}